\documentclass[12pt]{amsart}

\usepackage{amssymb, amscd}
\usepackage{latexsym,epsfig}
\usepackage[all]{xy}
\usepackage[pdf]{pstricks}
\usepackage{hyperref}

\numberwithin{equation}{section}
\def\today{\ifcase\month\or Jan\or Febr\or  Mar\or  Apr\or May\or Jun\or  Jul\or Aug\or  Sep\or  Oct\or Nov\or  Dec\or\fi \space\number\day, \number\year}

\newcommand{\CC}{\mathbb C}
\newcommand{\EE}{\mathbb E}

\newcommand{\GG}{\mathbb G}

\newcommand{\PP}{\mathbb P}
\newcommand{\QQ}{\mathbb Q}

\newcommand{\ZZ}{\mathbb Z}

\newcommand{\ord}{\rm ord}

\newcommand\langepijl[1]{\buildrel {#1} \over \longrightarrow}
\newcommand{\Sym}{{\mathrm{Sym}}}

\def\leq{\leqslant}
\def\geq{\geqslant}

\numberwithin{equation}{section}
\newtheorem{theorem}{Theorem}[section]
\newtheorem{observation}[theorem]{Observation}
\newtheorem{lemma}[theorem]{Lemma}
\newtheorem{proposition}[theorem]{Proposition}

\newtheorem{criterion}[theorem]{Criterion}

\newtheorem{definition-lemma}[theorem]{Definition-Lemma}
\theoremstyle{definition}
\newtheorem{definition}[theorem]{Definition}
\newtheorem{example}[theorem]{Example}

\theoremstyle{remark}
\newtheorem{remark}[theorem]{Remark}

\begin{document}

\title[Tautological Modular Forms ]{Tautological Modular Forms of Level Two\\ and Degree Two}
\begin{abstract}
We show how to use divisors on the projectivized Hodge bundle to construct special 
vector-valued modular forms and then apply invariant theory to
construct all vector-valued Siegel modular forms of level two and degree two. Thus we
construct all modular forms 
in terms of certain basic modular forms that
 are intimately connected to the moduli of curves of genus two.
\end{abstract}

\author{Fabien Cl\'ery}
\address{Department of Mathematics, Loughborough University, UK}
\email{cleryfabien@gmail.com}

\author{Gerard van der Geer}
\address{Korteweg-de Vries Instituut, Universiteit van
Amsterdam, Postbus 94248,
1090 GE  Amsterdam, The Netherlands}
\email{g.b.m.vandergeer@uva.nl}

\maketitle

\begin{section}{Introduction}\label{sec-intro}
Moduli spaces of curves and abelian varieties come with a Hodge bundle ${\EE}$.
Sections of a power $L^{\otimes k}$ of $L=\det({\EE})$ on such a moduli space $M$
are called modular forms of weight $k$ if these sections extend over an appropriate smooth compactification $\overline{M}$
of $M$. Besides these scalar-valued modular forms we can also consider vector-valued
modular forms defined as sections of vector bundles obtained by applying a Schur functor
to ${\EE}$, like ${\rm Sym}^j({\EE})\otimes L^{\otimes k}$. Sometimes all such modular forms can be constructed from a few
modular forms that are intimately connected with the moduli interpretation of the moduli
space. We call such modular forms tautological. 

A first example is provided by the moduli space $\mathcal{A}_1$ of elliptic curves over ${\CC}$.
In this case the ring of modular forms is generated by the Eisenstein series $E_4$ and $E_6$
of weight $4$ and $6$. Up to a scalar factor these two forms appear as the coefficients
$g_2$ and $g_3$ in the equation of the universal elliptic curve
$$
y^2=4\, x^3-g_2 x -g_3 \, .
$$

A second example is provided by the moduli space $\mathcal{A}_2$ of 
principally polarized abelian surfaces over an algebraically closed field. 
It contains the moduli space $\mathcal{M}_2$ of curves of genus $2$ as an open subset via the Torelli map.
In characteristic not $2$ the universal curve over $\mathcal{M}_2$ can be written as
$y^2=f$ with $f$ a polynomial of degree $6$. In joint work \cite{CFG1,CFG2} with Carel Faber the polynomial $f$ was interpreted as
a meromorphic vector-valued Siegel modular form $\chi_{6,-2}$ of weight $(6,-2)$ 
and its discriminant ${\rm discr}(f)$ 
 as a modular form $\chi_{10}$ of weight $10$. We showed that all  
vector-valued Siegel modular forms of degree~$2$ and level $1$ can be expressed in
terms of these two modular forms.  Invariant theory of binary sextics tells us 
which expressions we need. A variant also works in characteristic~$2$.

A third example is provided by the moduli space $\mathcal{M}_3$ of curves of genus $3$.
Here the universal non-hyperelliptic curve of genus $3$ is given as a quartic curve
$f(x,y,z)=0$ in ${\PP}^2$. In this case we interpreted in \cite{CFG3} the polynomial $f$
as a vector-valued meromorphic Teichm\"uller modular form $\chi_{4,0,-1}$ of weight $(4,0,-1)$,
that is, as 
a section of ${\rm Sym}^4({\EE})\otimes \det({\EE})^{-1}$ on $\overline{\mathcal{M}}_3$. 
There is also a modular form
$\chi_9$ of weight $9$ corresponding to the divisor in $\overline{\mathcal{M}}_3$ defined
by the hyperelliptic locus. We showed that all Teichm\"uller modular forms on 
$\overline{\mathcal{M}}_3$ can be expressed in $\chi_{4,0,-1}$ and $\chi_{9}$. 
A consequence is that all Siegel modular forms of degree $3$ and level $1$ can be expressed
in the these two modular forms.  In this case the invariant theory of ternary quartics
provides the expressions that one needs.

\smallskip

In view of the usefulness of such tautological modular forms 
one might try to find other examples of such forms.
In this paper we construct such modular forms for the case of 
Siegel modular forms of level $2$ and degree $2$. 
The moduli space $\mathcal{A}_2[2]$ of principally polarized abelian surfaces with a
full level $2$ structure admits an action by the symmetric group $\mathfrak{S}_6\cong
{\rm Sp}(4,{\ZZ}/2{\ZZ})$. We can fix this action by requiring
 that the induced action on the moduli space $\mathcal{M}_2[2]$ 
of curves of genus $2$ together with marked Weierstrass points is by 
the permutation of the six Weierstrass points. 

\smallskip

By interpreting the compactification $\widetilde{\mathcal{A}}_2[2]$ as the closure $\overline{\mathcal{M}}_2[2]$ of $\mathcal{M}_2[2]$ we
construct `tautological' vector-valued modular forms for the six sections of the
projectivized Hodge bundle defined by the Weierstrass points
on the universal curve over $\mathcal{M}_2[2]$. It turns out that these forms are related to
the gradients of the six odd theta functions.

\smallskip
Furthermore, the interpretation of $\mathcal{M}_2[2] \subset \mathcal{A}_2[2]$ 
as a stack quotient for ${\rm GL}_2$ allows us to interpret vector-valued modular forms
on $\mathcal{A}_2[2]$ as covariants for the action of ${\rm GL}_2$ on $V^{\oplus 6}$, the space of
six linear homogeneous forms in two variables. Classical invariant theory by Gordan
provides a description of generators of the bigraded ring $\mathcal{C}(V^{\oplus 6})$ of covariants.
This interpretation gives us a homomorphism from the bigraded ring  $\mathcal{R}(\mathcal{A}_2[2]) 
=\oplus_{j,k} M_{j,k}(\Gamma_2[2])$ of vector-valued modular forms
to the ring of covariants
$$
\mu :\mathcal{R}(\mathcal{A}_2[2]) \longrightarrow \mathcal{C}(V^{\oplus 6})
$$
that factors through  the subring $\mathcal{C}'(V^{\oplus 6}) \subset \mathcal{C}(V^{\oplus 6})$ generated by covariants
that satisfy a condition on the degree in all six linear forms. We thus get homomorphisms
$$
\mathcal{R}(\mathcal{A}_2[2]) \xrightarrow{\mu} \mathcal{C}^{\prime}(V^{\oplus 6})\xrightarrow{\nu}
\mathcal{R}(\mathcal{M}_2[2])= \mathcal{R}(\mathcal{A}_2[2])[1/\chi_{5}] \, .
$$
with $\nu \circ \mu$ the identity on the ring of vector-valued modular forms $\mathcal{R}(\mathcal{A}_2[2])$ 
of level $2$ and $\chi_5$ the square root of ${\rm disc}(f)$. The map $\mu$ embeds the ring
$\mathcal{R}(\mathcal{A}_2[2])$, that is not finitely generated, into a finitely generated ring. 

Using the tautological modular forms constructed geometrically we can describe the homomorphism $\nu$
explicitly by substitution in covariants. Since we land in a ring of meromorphic modular forms
with possible poles along $\mathcal{A}_2[2]-\mathcal{M}_2[2]$ we need a criterion of holomorphicity.
We give an algebraic criterion and thus via $\mu$ we can identify
the ring of vector-valued Siegel modular forms of degree $2$ 
explicitly as a subring of $\mathcal{C}'(V^{\oplus 6})$:
$$
\mathcal{R}(\mathcal{A}_2[2]) \cong \{ c \in \mathcal{C}'(V^{\oplus 6}): {\rm ord}(c)\geq 0 \}\, .
$$

Our tautological modular forms thus come in three descriptions: defined by divisors, 
defined by coefficients or factors of the binary sextic, and by theta functions and their derivatives.

\smallskip

Using the $\mathfrak{S}_6$-cover $\mathcal{A}_2[2] \to \mathcal{A}_2$ we can describe in a similar
way the vector-valued modular forms for all intermediate levels between $1$ and $2$.
A good example is provided by
the moduli space $\mathcal{A}_2[w]= \mathcal{A}_2[2]/\mathfrak{S}_5$ that
contains the moduli space $\mathcal{M}_2[w]$ of curves of genus $2$ with a marked Weierstrass 
point as an open subspace. For a curve of genus $2$ with a marked Weierstrass
point the defining polynomial $f$ of degree $6$ splits off a linear term.
We show that in this case there exist vector-valued modular forms $\chi_{5,-1}$ and $\chi_{1,-1}$
with a character such that $\chi_{6,-2}=\chi_{5,-1}\chi_{1,-1}$ corresponds to 
the splitting of $f$ as a product of a degree $5$ term and a linear term. We show that all 
modular forms of level $\mathcal{A}_2[w]$ can be expressed in the forms 
$\chi_{5,-1}$, $\chi_{1,-1}$ and $\chi_5$, the square root of $\chi_{10}$.
In this case the resulting homomorphism
$$
\nu: \mathcal{C}'(V^{\oplus 5} \oplus V) \to \mathcal{R}(\mathcal{A}_2[w])[1/\chi_5]
$$
does not depend on the splitting.
A similar game can be played for all the intermediate levels between level $1$
and level~$2$ and we thus can use covariants of binary forms of smaller degree.

All of this can be made completely explicit. We can describe the vector-valued 
modular forms that we need  explicitly using 
the ten even theta constants and the gradients of the six odd theta functions.
The homomorphism from the ring of covariants is obtained by substituting  such modular forms
in covariants and we thus obtain the Fourier
expansions of the modular forms from those of the thetas.

Finally, we notice that the rings of vector-valued modular forms of degree $2$
are not finitely generated as shown by Grundh \cite[p.\ 234]{BGHZ}. Therefore the description
of such rings of modular forms inside finitely generated rings of covariants
is a good compensation for the lack of a finite set of generators. 
Moreover, this construction method of modular forms is ready made for 
machine construction of modular forms.
We hope to extend  this also to some extent to genus $3$.

\bigskip

{\sl Acknowledgement}:
The second author thanks YMSC of Tsinghua University in Beijing and MPIM in Bonn
for providing an excellent working environment.
\end{section}
\tableofcontents
\begin{section}{Moduli stacks}
The idea behind the method is the description of a moduli space as a stack quotient
and the Hodge bundle in terms of an equivariant bundle over the stack quotient.

The basic example is the moduli space $\mathcal{M}_2$ over an algebraically closed
field $k$ of characteristic not $2$. A curve of genus $2$ is hyperelliptic and 
the canonical map allows a description of the curve as a double cover of ${\PP}^1$
ramified at six points. Thus we may write it in affine form as $y^2=f$ with $f\in k[x]$ a polynomial of
degree $6$ with non-vanishing discriminant. If we use projective coordinates we may write
$f$ as a homogenous polynomial in $x_1,x_2$, say $f \in {\rm Sym}^6(V)$ with $V$
the vector space with basis $x_1,x_2$. Changing the basis changes $f$ by the action of
${\rm GL}(V)={\rm GL}_2$. We modify this action to take the form 
$$
f(x_1,x_2) \mapsto (ad-bc)^{-2} f(ax_1+bx_2,cx_1+dx_2)\quad \text{\rm under 
$\left( \begin{matrix} a &b \\ c & d \\ \end{matrix} \right) \in {\rm GL}_2$ }
$$
and view it as the natural action of ${\rm GL}_2$ on the irreducible
representation $V_{6,-2}={\rm Sym}^6(V)\otimes \det(V)^{-2}$.
 Then the stabilizer of a generic $f$ is of order $2$, and generated by $-1$.
We can view $V_{6,-2}$ as the space of binary sextics with a slightly
modified action of ${\rm GL}_2$

The moduli space $\mathcal{M}_2$ now allows the description as a stack quotient
$$
\alpha: [V_{6,-2}^0/{\rm GL}_2] \xrightarrow{\sim} \mathcal{M}_2\, ,
$$
where $V_{6,-2}^0 \subset V_{6,-2}$ is the Zariski open subset of polynomials with
non-vanishing discriminant. The important point is that the pull back of the Hodge
bundle ${\EE}$ is the equivariant bundle $V$. As explained in \cite{CFG1} and exploited further
in \cite{CFG2}, pullback of a section of ${\rm Sym}^j({\EE})\otimes \det({\EE})^k$
defines a covariant of binary sextics of degree $d=k+j/2$ and order $j$. Conversely, a covariant
of degree $k+j/2$ and order $j$ defines a section of ${\rm Sym}^j({\EE})\otimes \det({\EE})^k$
on $\mathcal{M}_2=\mathcal{A}_2-\mathcal{A}_{1,1}$ with $\mathcal{A}_{1,1}$
the locus of products of elliptic curves. Scalar-valued modular forms define invariants. 
The ring of scalar-valued modular forms
thus embeds as a subring in the ring of invariants as Igusa already noticed, but he used
theta functions to establish that fact, see \cite{Igusa1964}. The ring 
$\mathcal{R}(\mathcal{A}_2)$
of vector-valued modular forms
embeds as a non-finitely generated subring of the finitely generated ring $\mathcal{C}({\rm Sym}^6(V))$
of covariants
of binary sextics, while the ring of covariants can be viewed as a subring 
of the ring of vector-valued modular forms on $\mathcal{M}_2[2]$.
The homomorphism from the ring of covariants $\mathcal{C}(2,6)$ of
binary sextics sends the universal binary sextic $f=\sum_{i=0}^6 a_i x_1^{6-i}x_2^i$ 
to the meromorphic modular form
$\chi_{6,-2}$ on $\mathcal{A}_2$ 
and consists of replacing the coefficients $a_i$ by the coefficients $\alpha_i$
of the modular form $\chi_{6,-2}$. The modular form $\chi_{6,-2}$ has a simple pole
along $\mathcal{A}_{1,1}$ caused by the pole of $\alpha_3$ and allows determination of
holomorphicity of a modular form determined by a covariant.

\bigskip

It is our aim to generalize this to the modular forms on intermediate levels between
level $1$ and level $2$. We denote by $\mathcal{A}_2[2]$ the moduli space of
principally polarized abelian surfaces with a full level $2$ structure.
Over ${\CC}$ it is the space $\Gamma_2[2]\backslash \mathfrak{H}_2$ with
$\Gamma_2[2]$ the kernel of ${\rm Sp}(2,{\ZZ}) \to {\rm Sp}(2,{\ZZ}/2{\ZZ})$
and $\mathfrak{H}_2$ the Siegel upper half space of degree $2$.
The group ${\rm Sp}(2,{\ZZ}/2{\ZZ})$ acts on $\mathcal{A}_2[2]$ and on the 
Torelli image of the moduli space $\mathcal{M}_2[2]$ of curves of genus $2$ with a level $2$ structure.
The latter structure may be given by marking the six Weierstrass points on the curve.
The action of ${\rm Sp}(2,{\ZZ}/2{\ZZ})$ on the six Weierstrass points determines
an identification with the symmetric group $\mathfrak{S}_6$ on six letters. Then the intermediate
levels are given via the Galois action of $\mathfrak{S}_6$ on $\mathcal{A}_2[2]$.

We now look at full level $2$. We have a diagram of moduli stacks
$$
\begin{xy}
\xymatrix{
\mathcal{M}_2[2] \ar@{^{(}->}[r] \ar[d] & \mathcal{A}_2[2] \ar[d] \\
\mathcal{M}_2 \ar@{^{(}->}[r] & \mathcal{A}_2 \\
}
\end{xy}
$$
The group ${\rm Sp}(4,{\ZZ}/2{\ZZ})\cong \mathfrak{S}_6$ acts. Recall that $\mathfrak{S}_6$ possesses an outer automorphism. 
As just said, we can fix this isomorphism and hence the
action of $\mathfrak{S}_6$ by identifying
our stack $\mathcal{M}_2[2]$ with the stack $\mathcal{M}_{2,W}$, the moduli of curves of genus $2$ with six marked Weierstrass points.
The group $\mathfrak{S}_6$ naturally acts on the six Weierstrass points inducing an action on
$\mathcal{M}_2[2]$ and $\mathcal{A}_2[2]$. Note that the generic element $(C,p_1,\ldots,p_6)$
has a stabilizer of order $2$ generated by the hyperelliptic involution.

We know how to interpret $\mathcal{M}_2$ as a stack quotient for ${\rm GL}_2$ via
$$
\mathcal{M}_2 \xrightarrow{\sim} [ V_{6,-2}^0/{\rm GL}(V)]\, ,
$$
where $V=\langle x_1,x_2\rangle$ is a vector space with basis $x_1,x_2$ and $V_{6,-2}={\rm Sym}^6(V)\otimes \det(V)^{-2}$
with $V_{6,-2}^0$ denoting the Zariski open set of polynomials of degree $6$ with non-vanishing discriminant.
Then the stack $\mathcal{M}_2[2]$ allows the interpretation as a quotient stack as follows. Define
$$
\mathcal{P}= \{ (p_1,\ldots,p_6,f) \in ({\PP}^1)^6 \times V_{6,-2}^0: f(p_i)=0 \}\, .
$$
Note that ${\rm GL}_2$ acts on ${\PP}^1$ in the usual way via its quotient ${\rm PGL}_2$.
We then have the equivalence of stacks
$$
\varphi: [\mathcal{P} / {\rm GL}_2] \cong \mathcal{M}_2[2] \, .
$$
Note that the generic element $(p_1,\ldots,p_6,f)$ has as stabilizer the group of order $2$
generated by $-{\rm id}_V$.
The action of $\mathfrak{S}_6$
permutes the six roots of $f$.

\smallskip

The pull back of the Hodge bundle ${\EE}$ on $\mathcal{M}_2$ to $[\mathcal{P}/{\rm GL}_2]$ is the
equivariant vector bundle $V$. A scalar-valued modular form $F$ on $\mathcal{A}_2[2]$ of weight $k$ restricts to
$\mathcal{M}_2[2]$ and then defines an invariant of degree $d=k$ for the action of ${\rm GL}_2$ on $V^6$, that is,
it is given by a homogeneous function $\tilde{F}$ of degree $k$ in the coordinates $(\alpha_i:\beta_i)$, $i=1,\ldots,6$,
of $({\PP}^1)^6$ that is invariant under the action of ${\rm SL}_2$. Similarly, a vector-valued modular form $F$
on $\mathcal{A}_2[2]$ of weight $(j,k)$ restricts to a covariant of bidegree $(d,b)$ with $j=b$ and $k=d-b/2$;
that is, it is defined as a polynomial that is homogeneous of degree $d$ in the $(\alpha_i:\beta_i)$ and homogeneous of
degree $j$ in $x_1,x_2$ and that is invariant under the action of~${\rm SL}_2$.

The action of $\mathfrak{S}_6$ induces an action on spaces of covariants and modular forms.
For the reader's convenience we give a table with the dimensions of the irreducible representations of
$\mathfrak{S}_6$. These representations are given by the partitions of $6$. Please keep in mind
 that $\mathfrak{S}_6$ allows an outer automorphism, see Remark \ref{outerauto}.

\begin{footnotesize}
\smallskip
\vbox{
\bigskip\centerline{\def\quad{\hskip 0.6em\relax}\def\quod{\hskip 0.5em\relax }
\vbox{\offinterlineskip
\hrule\halign{&\vrule#&\strut\quod\hfil#\quad\cr
height2pt&\omit&&\omit&&\omit&&\omit&&\omit&&\omit&&\omit&&\omit&&\omit&&\omit &
& \omit && \omit &\cr
&$P$ && $[6]$ && $[5,1]$ && $[4,2]$ && $[4,1^2]$ && $[3^2]$ && $[3,2,1]$ && $[3,1^3]$ && $[2^3]$ && $[2^2,1^2]$ && $[2,1^4]$&& $[1^6]$ &\cr
\noalign{\hrule}&$\dim$ && $1$ && $5$ && $9$ && $10$ && $5$ && $16$ && $10$ && $5$ && $9$ && $5$ && $1$ &\cr
} \hrule}
}}
\end{footnotesize}

\end{section}
\begin{section}{Theta series with characteristics and their gradients}\label{theta}
In this auxiliary section we describe the modular forms defined by theta series that we need later.
Most of the results in this section can be found in \cite[Chapter V]{IgusaBook}.
For any $(\tau,z)\in \mathfrak{H}_2\times {\CC}^2$ and
$
\left[
\begin{smallmatrix}
\mu \\ \nu
\end{smallmatrix}
\right]=
\left[
\begin{smallmatrix}
\mu_1 & \mu_2 \\ \nu_1 & \nu_2
\end{smallmatrix}
\right]
$
with $(\mu_1,\mu_2)$ and $(\nu_1,\nu_2)$ in ${\ZZ}^2$, the
theta series with characteristics is defined as follows
\begin{equation}\label{ThetaCharac}
\vartheta_{
\left[
\begin{smallmatrix}
\mu \\ \nu
\end{smallmatrix}
\right]}
(\tau,z)=\hspace{-15pt}
\sum_{n=(n_1,n_2)\in {\ZZ}^2}
e^{\pi i (n+\mu/2)
\big(
\tau (n+\mu/2)^t
+2(z+\nu/2)^t
\big)}.
\end{equation}
For any
$
\left[
\begin{smallmatrix}
\mu' \\ \nu'
\end{smallmatrix}
\right]=
\left[
\begin{smallmatrix}
\mu'_1 & \mu'_2 \\ \nu'_1 & \nu'_2
\end{smallmatrix}
\right]
$
with $(\mu'_1,\mu'_2)$ and $(\nu'_1,\nu'_2)$ in ${\ZZ}^2$, a direct computation shows that
\begin{equation}\label{AdditionCharac}
\vartheta_{
\left[
\begin{smallmatrix}
\mu \\ \nu
\end{smallmatrix}
\right]+
2\left[
\begin{smallmatrix}
\mu' \\ \nu'
\end{smallmatrix}
\right]}
(\tau,z)=(-1)^{\nu'\,\mu^t}
\vartheta_{
\left[
\begin{smallmatrix}
\mu \\ \nu
\end{smallmatrix}
\right]}
(\tau,z)
\end{equation}
and this allows us to only take $(\mu_1,\mu_2)$ and $(\nu_1,\nu_2)$ in $\{0,1\}^2$ in~\eqref{ThetaCharac} and this is what we do from now on.
According to the parity of $\mu\nu^t=\mu_1\nu_1+\mu_2\nu_2$, we call the characteristics
$\left[\begin{smallmatrix} \mu \\ \nu \end{smallmatrix}\right]$  even or odd. As a function of $z$, the function
$\vartheta_{\left[\begin{smallmatrix}\mu \\ \nu\end{smallmatrix}\right]}$ has the same parity as its characteristic
$\left[\begin{smallmatrix}\mu \\ \nu\end{smallmatrix}\right]$.
There are ten even characteristics that we order as follows
\begin{align*}
& n_1=\left[\begin{matrix} 0 & 0 \cr 0 & 0 \cr\end{matrix}\right],
n_2=\left[\begin{matrix} 0 & 0 \cr 0 & 1 \cr\end{matrix}\right],
n_3=\left[\begin{matrix} 0 & 0 \cr 1 & 0 \cr\end{matrix}\right],
n_4=\left[\begin{matrix} 0 & 0 \cr 1 & 1 \cr\end{matrix}\right],
n_5=\left[\begin{matrix} 0 & 1 \cr 0 & 0 \cr\end{matrix}\right], \\
& n_6=\left[\begin{matrix} 0 & 1 \cr 1 & 0 \cr\end{matrix}\right],
n_7=\left[\begin{matrix} 1 & 0 \cr 0 & 0 \cr\end{matrix}\right],
n_8=\left[\begin{matrix} 1 & 0 \cr 0 & 1 \cr\end{matrix}\right],
n_9=\left[\begin{matrix} 1 & 1 \cr 0 & 0 \cr\end{matrix}\right],
n_{10}=\left[\begin{matrix} 1 & 1 \cr 1 & 1 \cr\end{matrix}\right]
\end{align*}
and six odd ones
\[
m_1=\left[\begin{matrix} 0 & 1 \cr 0 & 1 \cr\end{matrix}\right], \,
m_2=\left[\begin{matrix} 0 & 1 \cr 1 & 1 \cr\end{matrix}\right], \,
m_3=\left[\begin{matrix} 1 & 0 \cr 1 & 0 \cr\end{matrix}\right], \,
m_4=\left[\begin{matrix} 1 & 0 \cr 1 & 1 \cr\end{matrix}\right], \,
m_5=\left[\begin{matrix} 1 & 1 \cr 0 & 1 \cr\end{matrix}\right], \,
m_6=\left[\begin{matrix} 1 & 1 \cr 1 & 0 \cr\end{matrix}\right].
\]
We ordered these characteristics in the same way 
as in \cite[Section 3]{CvdGG}, that is, lexicographically.
For an even characteristic $n_i$, let us simply denote the theta constant
$\vartheta_{n_i}$ by
\[
\vartheta_i(\tau)=\vartheta_{n_i}(\tau,0).
\]
An even theta characteristic can be written in two ways as a sum of three different odd theta
characteristics. This establishes a bijection between the even theta characteristics and the
partitions of $\{1,2,\ldots,6\}$ in two triples:
$$
\begin{matrix}
n_1 \, (146)(235) & n_2 \, (136)(245) & n_3 \, (135)(246) & n_4 \, (145)(236) & n_5 \, (134)(256)\\
n_6 \, (156)(234) & n_7 \, (123)(456) & n_8 \, (124)(356) & n_9 \, (126)(345) & n_{10} \, (125)(346)\\
\end{matrix}
$$
The complement $\mathcal{A}_{1,1}[2]$ of $\mathcal{M}_2[2]$ in $\mathcal{A}_2[2]$ is a divisor and consists of ten
irreducible components with each component corresponding to an even theta characteristic and also to a
partitition $(abc)(def)$ of $\{1,2,\ldots,6\}$. (These irreducible divisors are denoted $H_{\pi}$ for a partition $\pi$, see
Section 4.)

\bigskip
We need the Fourier expansions of the gradients $G_i$ of the odd theta functions.
A direct manipulation of (\ref{ThetaCharac}) leads to
\[
\vartheta_{
\left[
\begin{smallmatrix}
\mu \\ \nu
\end{smallmatrix}
\right]}
(\tau,z)=
\sum_{(n_1,n_2)\in {\ZZ}^2}
i^{(2n_1+\mu_1)\nu_1+(2n_2+\mu_2)\nu_2}
\zeta_1^{2n_1+\mu_1}
\zeta_2^{2n_2+\mu_2}
Q_1^{(2n_1+\mu_1)^2}
Q_{12}^{(2n_1+\mu_1)(2n_2+\mu_2)}
Q_2^{(2n_2+\mu_2)^2}
\]
where for
$
\tau=
\left(
\begin{smallmatrix}\tau_1 & \tau_{12} \\ \tau_{12} & \tau_2\end{smallmatrix}
\right)\in \mathfrak{H}_2$
and
$
z=(z_1,z_2)\in {\CC}^2
$,
we set
\[
Q_1=e^{\pi i \tau_1/4},\quad
Q_{12}=e^{\pi i \tau_{12}/2},\quad
Q_2=e^{\pi i \tau_2/4},\quad
\zeta_1=e^{\pi i z_1},\quad
\zeta_2=e^{\pi i z_2}.
\]
We easily deduce the Fourier expansion of the ten even theta constants
\begin{equation}\label{FE_Even_Theta}
\vartheta_{
\left[
\begin{smallmatrix}
\mu \\ \nu
\end{smallmatrix}
\right]}(\tau)
=
i^{\mu_1\nu_1+\mu_2\nu_2}
\hspace{-15pt}
\sum_{n=(n_1,n_2)\in {\ZZ}^2}
\hspace{-15pt}
(-1)^{n_1\nu_1+n_2\nu_2}
Q_1^{(2n_1+\mu_1)^2}
Q_{12}^{(2n_1+\mu_1)(2n_2+\mu_2)}
Q_2^{(2n_2+\mu_2)^2} \, .
\end{equation}
We put 
$$
\chi_5= -2^{-6} \prod_{i=1}^{10} \vartheta_i = (Q_{12}^2-Q_{12}^{-2})(Q_1Q_2)^4+\ldots \, .
$$
This is a modular form of level $1$ with a (quadratic) character.

For an odd characteristic $m_i$, we denote by $G_i$ the (normalised) gradient of the theta series $\vartheta_{m_i}$ evaluated at $(z_1,z_2)=0$
\begin{equation}\label{GradTheta}
G_i(\tau)=
\left[\begin{smallmatrix} 
G_{i,1}(\tau)\\G_{i,2}(\tau)
\end{smallmatrix}\right]=
\frac{1}{\pi i}\nabla\vartheta_{m_i}(\tau)=
\frac{1}{\pi i}\left[\begin{smallmatrix} 
\partial \vartheta_{m_i} /\partial z_1\\
\partial \vartheta_{m_i} /\partial z_2
\end{smallmatrix}\right](\tau,0).
\end{equation}
We normalised these gradients to avoid powers of $\pi$ in the sequel; the Fourier expansions of their components are given by
 \begin{equation}\label{FE_Odd_Theta}
G_{i,j}(\tau)= i^{\mu_1\nu_1+\mu_2\nu_2}
\hspace{-14pt}
\sum_{n=(n_1,n_2)\in {\ZZ}^2}
\hspace{-8pt}
(-1)^{n_1\nu_1+n_2\nu_2}
(2n_j+\mu_j)
Q_1^{(2n_1+\mu_1)^2}
Q_{12}^{(2n_1+\mu_1)(2n_2+\mu_2)}
Q_2^{(2n_2+\mu_2)^2}
\end{equation}
These gradients are sections of ${\EE}\otimes \det({\EE})^{1/2}$ on the moduli space
$\mathcal{A}_2[4,8]$ of abelian surfaces with a level $(4,8)$-structure.

\bigskip

One way to retrieve a binary sextic defining a curve $C$ of genus $2$ from its Jacobian ${\rm Jac}(C)$ 
(in characteristic not $2$)
is by looking at the sixteen symmetric theta divisors on ${\rm Jac}(C)$.
There are precisely six such theta divisors passing through the origin. The tangent lines
to these six theta divisors define six lines in the $2$-dimensional tangent space
to ${\rm Jac}(C)$ at the origin. The six odd theta functions have as their divisors
these six theta divisors and their gradients define these six lines in the tangent space.
So up to a power of $\chi_5$ these six $G_i$ define the tautological form $\chi_{6,-2}$.
This is in line with the fact that $\chi_5\,  {\rm Sym}^6(G_1,\ldots,G_6)$ is a cusp form of weight $(6,8)$
as shown in \cite{CvdGG}. Let $S_{j,k}(\Gamma_2)$ be the space of cusp forms of degree $2$
on $\Gamma_2={\rm Sp}(4,{\ZZ})$. We also write $M_{j,k}(\Gamma_2[2])$ and $S_{j,k}(\Gamma_2[2])$
for the spaces of modular forms and cusp forms on $\mathcal{A}_2[2]$, that is, on the congruence subgroup
$\Gamma_2[2]$ of level $2$.

\begin{proposition}
We have
$\chi_5\, {\rm Sym}^6(G_1,\ldots,G_6) \in S_{6,8}(\Gamma_2)$.
\end{proposition}
We know that $\dim S_{6,8}(\Gamma_2)=1$.
In particular we see that ${\rm Sym}^6(G_1,\ldots,G_6)/\chi_5$ equals a multiple
of $\chi_{6,-2}$.

As these six lines are given by their Pl\"ucker coordinates, we set
$$
\tilde{p}_{ij}=G_i\wedge G_j=G_{i1}G_{j2}-G_{i2}G_{j1} \quad (1\leq i < j \leq 6)\, .
$$
It is appropriate to describe the action of the symmetric group $\mathfrak{S}_6$.
The action of $\sigma\in \mathfrak{S}_6$ is via $\sigma(\tilde{p}_{ij})=\tilde{p}_{\sigma(i) \sigma(j)}$.
The action on the $\vartheta_i=\vartheta_{n_i}$ is determined by the action on the 
characteristics and equivalently, on the partitions $n_i$ up to an eighth root of $1$.

We know that an unordered pair of odd theta characteristics $(a,b)$
determines a quadruple of even theta characteristics, namely the four even theta characteristics
such that the corresponding partition of $6$ is
of the form $(abc)(def)$,
see above or \cite[Lemma 3.1]{CvdGG}.
Accordingly, by a small computation we find
$$
\tilde{p}_{12}= \vartheta_7\vartheta_8\vartheta_9 \vartheta_{10} \, .
$$
This is in agreement with \cite[Lemma 5.1]{F-SM2015}.
Note that for the action of $\mathfrak{S}_6\cong {\rm Sp}(4,{\ZZ}/2{\ZZ})$ we have
$ \sigma(\tilde{p}_{12})= \vartheta_7\vartheta_8\vartheta_9 \vartheta_{10}{|_2 \sigma^{-1}}$
for $\sigma \in \mathfrak{S}_6$ where the $|_2\sigma^{-1}$ denotes the usual slash in weight $2$.
We observe that 
$$
\prod_{1 \leq i < j \leq 6} \tilde{p}_{ij}= - 2^{36} \chi_5^6\, .
$$

\end{section}
\begin{section}{Modular forms from geometry}
We now show how to obtain scalar-valued and vector-valued modular forms 
from geometry, more precisely from effective divisors as in \cite{vdG-K}.
The Hodge bundle ${\EE}$ of rank~$2$ over $\mathcal{A}_2[2]$ extends over $\tilde{\mathcal{A}}_2[2]$. 
We will use the same notation for the extension. Its projectivization ${\PP}({\EE})$ is a ${\PP}^1$-bundle
$\varpi: {\PP}({\EE})\to \tilde{\mathcal{A}}_2[2]$.
The six Weierstrass points in the universal curve over $\mathcal{M}_2[2]$ define six sections of
${\PP}({\EE})$ over $\mathcal{M}_2[2]$ via the canonical morphism. We can take their closures and thus obtain six divisors $W_i$
with $i=1,\ldots,6$ in ${\PP}({\EE})$ over $\tilde{\mathcal{A}}_2[2]$. 
Their divisor classes, also denoted $W_i$, can be written as
$$
W_i= h+\varpi^*(A_i)
$$
with $h=c_1(\mathcal{O}_{{\PP}({\EE})}(1))$ the first Chern class and $A_i$ a divisor class on $\tilde{\mathcal{A}}_2[2]$.
We use Grothendieck's interpretation for ${\PP}({\EE})$; of course ${\PP}({\EE})\cong {\PP}({\EE}^{\vee})$, but the $\mathcal{O}(1)$
depends on the interpretation.

The Picard group with rational coefficients of  $\tilde{\mathcal{A}}_2[2]$ has rank $16$ and is generated by the
$\lambda=c_1({\EE})$ and the $15$ boundary components $D_{ij}$ of $\tilde{\mathcal{A}}_2[2]$ with $1\leq i < j \leq 6$.
The action of $\sigma \in \mathfrak{S}_6$ on the $D_{ij}$ is by $D_{ij} \mapsto D_{\sigma(i) \sigma(j)}$.

The closure of the locus $\mathcal{A}_{1,1}[2]$ of products of elliptic curves consists 
of ten irreducible components $H_{\pi}$ of $\mathcal{A}_{1,1}[2]$ corresponding to the ten partitions $\pi=(abc)(def)$ of 
$\{1,2,\ldots,6\}$.

In line with the moduli of curves interpretation we will use the notation
$$
\delta_0=\sum_{1 \leq i < j \leq 6} D_{ij}, \quad \delta_1=\sum_{\pi} H_{\pi} \, .
$$

The effective divisor $H_{\pi}$ can be expressed in our basis as follows.
\begin{equation} \label{tenrelations}
4\, H_{(abc)(def)}= 2\, \lambda - \left( D_{ab}+D_{ac}+D_{bc}+D_{de}+D_{df}+D_{ef} \right) \, . 
\end{equation}
This may be deduced from the Picard group of $\overline{\mathcal{M}}_{0,6}$ or just by looking
at the zero divisor of the fourth power of an even theta characteristic, see Section \ref{theta}.

The divisor of the modular form $\chi_5$
gives the relation
$$
5\, \lambda = \delta_1 +\delta_0\, .
$$
Define for $i=1,\ldots,6$ a splitting of $\sum D_{ij}$ in a sum of six and a sum of nine divisors
$$
\Delta_i= \sum_{1\leq j \leq 6, j\neq i} D_{ij} \qquad \text{\rm and} \qquad  \Delta_i^{\prime}= \delta_0 -\Delta_i
$$

\begin{lemma}
We have in ${\rm Pic}_{\QQ}(\tilde{\mathcal{A}}_2[2])$ that 
$A_i=\lambda/2 - \alpha \Delta_i - \alpha^{\prime} \Delta_i^{\prime}$ with $\alpha+2 \alpha'=1/2$. 
\end{lemma}
\begin{proof}
Since $A_i$ is invariant under the subgroup $\mathfrak{S}_5$ of $\mathfrak{S}_6$ fixing $i$ and $\Delta_i$ and $\Delta_i^{\prime}$
generate the $\mathfrak{S}_5$-invariant subspace of the space generated by the $D_{ij}$, we see that $A_i$ is equivalent to
$n\lambda - \alpha \Delta_i - \alpha^{\prime} \Delta_i^{\prime}$ for some $\alpha,\alpha'$ and $n$. 
We know by \cite[Section 2]{vdG-K} 
that
$$
A_1+\ldots + A_6= 8\lambda -2 \delta_0-\delta_1= 3\lambda-\delta_0\, ,
$$
where we note that the coefficient $2$ of $\delta_0$ is due to ramification when going from level $1$ to $2$.
We see that $6n=3$ and because $\sum_i \Delta_i= 2\delta_0$, $\sum_i \Delta_i'=4\delta_0$ we get $\alpha + 2 \alpha'=1/2$.
\end{proof}
Now we use the fact that 
$\varpi_*(\mathcal{O}_{{\PP}({\EE})}(j))\cong {\rm Sym}^j({\EE})$. 
\begin{observation} \label{opmerking}
If we have an effective divisor $F$
on ${\PP}({\EE})$ with divisor class
\begin{equation}
F= j \, h +\varpi^*(k\, \lambda - \sum_{ij} m_{ij} D_{ij})
\end{equation}
with non-negative integers $j$, $k$ and $m_{ij}$, 
then the image under $\varpi_*$ of the canonical section $1$ of $\mathcal{O}_{\PP({\EE})}(j)$ is a section
of ${\rm Sym}^j({\EE}) \otimes \det({\EE})^{\otimes k}$ vanishing with multiplicity $m_{ij}$ on $D_{ij}$.
In fact, the non-negativity of the $m_{ij}$ is not needed by the Koecher Principle.
\end{observation}
We know by \cite{vdG-K} that the divisor $\sum_{i=1}^6 W_i$ on ${\PP}({\EE})$ defines the modular 
form $\chi_{6,3}={\rm Sym}^6(G_1,\ldots,G_6)$
with $G_i$ the gradient of the odd theta function $\vartheta_{m_i}$ as defined in \cite{CvdGG}, see above in Section \ref{theta}.
We also know by \cite[Prop.\ 23.2]{CvdGG} that ${\rm Sym}^4(G_i)$ defines a modular form of weight $(4,2)$. This implies that the $\alpha$ and $\alpha'$
occurring in 
$A_i=\lambda/2 - \alpha \Delta_i - \alpha^{\prime} \Delta_i^{\prime}$ satisfy $4\alpha \in {\ZZ}$, $4\alpha' \in {\ZZ}$.

\begin{proposition}\label{Wiclass} In ${\rm Pic}_{\QQ}({\PP}({\EE}))$ we have $W_i= h+\varpi^*( \lambda/2 -\Delta'_i/4)$.
\end{proposition}
\begin{proof} We know that ${\rm Sym}^2(G_i)$ is not a modular form on $\Gamma_2[2]$. 
In fact, $\dim M_{2,1}(\Gamma_2[2])=0$.
In order that $4\, W_i$ defines a modular form, but $2\, W_i$ does not, we need in view of Observation 
\ref{opmerking} that
$$
4\, W_i=4\, h + \varpi^*(2\, \lambda - 4\, \alpha\Delta_i - 4\, \alpha'\Delta_i')
$$
has integral coefficients, but $2\, W_i=2\, h +  \varpi^*(\lambda - 2\alpha\Delta_i - 2\alpha'\Delta_i')$
does not, that is
$(4\, \alpha, 4\, \alpha')\in {\ZZ}\times {\ZZ}$, but $ (2\, \alpha,2\, \alpha') \not\in {\ZZ}\times {\ZZ}$.
This implies by the relation $\alpha+2\, \alpha'=1/2$ that $(\alpha,\alpha')=(0,1/4)$,
that is, $4\, A_i=2\, \lambda - \Delta_i'$.
\end{proof}
\begin{remark} i)
The expression ${\rm Sym}^2(G_i)$ defines a modular form of weight $(2,1)$
with a character on $\Gamma_2[2]$.
ii)  The modular form
${\rm Sym}^4(G_i)$ viewed as a section
of ${\rm Sym}^4({\EE}) \otimes \det({\EE})^{\otimes 2}$ vanishes simply on $\Delta_i'$.
\end{remark}
In view of Observation \ref{opmerking} we now turn our attention to explicit effective
divisors. 
We have for a given partition  ${\pi}=(abc)(def)$ the effective divisor class
$$
H_{\pi}= \frac{1}{4}\left(2\, \lambda- D_{ab}-D_{ac}-D_{bc}-D_{de}-D_{df}-D_{ef}\right)
$$
and for $i=1,\ldots,6$  the effective divisor class
$$
W_i=h+\varpi^*(\lambda/2-\Delta_i^{\prime}/4)\, .
$$
For non-negative integers
 $c_{\pi}$ and $d_i$
we consider the following effective divisor class on ${\PP}({\EE})$ over $\widetilde{\mathcal{A}}_2[2]$
\begin{equation}\label{effective divisorF}
F=\sum_{\pi} c_{\pi} \, \varpi^*(H_{\pi}) + \sum_{i=1}^6 d_i \, W_i\, .
\end{equation}
We can write it in terms of our basis $h,\varpi^*(\lambda), \varpi^*(D_{ij})$ of ${\rm Pic}_{\QQ}({\PP}({\EE}))$ as
\begin{equation}\label{Finbasis}
F=(\sum d_i)h+(\frac{1}{2}\sum_{\pi} c_{\pi}+\frac{1}{2}\sum d_i) \lambda - \sum r_{ij} D_{ij} \, , 
\end{equation}
where with $\{a,b,c,d,e,f\}=\{1,2,3,4,5,6\}$ the $r_{ij}$ are given by
\begin{equation}\label{rijcoeff}
r_{ab}= \frac{1}{4}\left(c_{(abc)(def)}+c_{(abd)(cef)}+c_{(abe)(cdf)}+c_{(abf)(cde)}+d_c+d_d+d_e+d_f   \right)\, .
\end{equation}

Using the fact that $\varpi_*(\mathcal{O}_{{\PP}({\EE})}(j))={\rm Sym}^j({\EE})$ we get the following result.
\begin{proposition}
Suppose the coefficients of the effective divisor $F$ in \ref{Finbasis} are integral and $r_{ij}\geq 0$ for all $(i,j)$. 
If  $s$ is the canonical section $1$ of $O(F)$ then $\varpi_*(s)\in M_{j,k}(\Gamma_2[2])$.
It is a section of ${\rm Sym}^j({\EE})\otimes \det({\EE})^k$
with $j=\sum_i d_i$ and $2 k=\sum_{\pi} c_{\pi}+\sum d_i$ vanishing with multiplicity $r_{ij}$
along $D_{ij}$.
\end{proposition}
\begin{example}
The divisor
$$
H_{(146)(235)}+H_{(136)(245)}+H_{(135)(246)}+H_{(145)(236)}+H_{(134)(256)}+H_{(156)(234)} 
+ W_1+W_2
$$
has class 
$$
2\, h +\varpi^*(4\, \lambda -\sum_{(ij)\neq (12)} D_{ij})
$$
and defines a modular form of weight $(2,4)$, namely ${\rm Sym}^2(G_1,G_2)\vartheta_1\vartheta_2\vartheta_3\vartheta_4\vartheta_5\vartheta_6$, 
cf.\ \cite[Example 16.2]{CvdGG}.
Similarly, the divisor
$$
2\, H_{(123)(456)}+ 2(W_1+W_2+W_3)=6\, h + \varpi^*(4\, \lambda -\delta_0- (D_{45}+D_{46}+D_{56}))
$$
defines a modular form of weight $(4,6)$. 
\end{example}

\end{section}

\begin{section}{Rings of covariants}
We write $V=\langle x_1,x_2\rangle$ for the vector space with basis $x_1,x_2$. The group ${\rm GL}(V)$
acts on $V^{\oplus 6}$, the space of six linear homogeneous forms $l_i=l_{i,1}x_1+l_{i,2}x_2$ ($i=1,\ldots,6$) 
in $x_1,x_2$. The ring of invariants $\mathcal{I}(V^{\oplus 6})$ consists of polynomials in the coordinates
$l_{i,1},l_{i,2}$ that are invariant under ${\rm SL}(V)\cong {\rm SL}_2$. By Gordan \cite{Gordan} this ring is generated
by the $15$ invariants 
$$
p_{ij}=l_{i,1}l_{j,2}-l_{i,2}l_{j,1}, \quad (1\leq i< j \leq 6),
$$ where the letter $p$ refers to Pl\"ucker. These satisfy the usual Pl\"ucker relations
$$
p_{ik}p_{jl}-p_{il}p_{jk}=p_{ij}p_{kl} \, .
$$
The ring $\mathcal{I}(V^{\oplus 6})=\oplus\,  \mathcal{I}_d(V^{\oplus 6})$ 
is graded by the degree $d$ in the coordinates $l_{i,1}$ and $l_{i,2}$.
We have (see \cite{Geyer})
$$
\dim \mathcal{I}_d(V^{\oplus 6})= (d+1)(d^2+2d+2)/2 \, .
$$
The ring of covariants $\mathcal{C}(V^{\oplus 6})$ consists of the polynomials
in the coordinates $\l_{i,1}$, $\l_{i,2}$ and $x_1,x_2$ that are invariant under ${\rm SL}_2$. 
It can be interpreted as the ring of invariants $\mathcal{I}(V^{\oplus 6} \oplus V^{\vee})$, 
see \cite[page 55]{Springer}. Gordan showed that it is generated by the forms
$l_i$ and the $p_{ij}$.

This ring is bigraded 
$\mathcal{C}(V^{\oplus 6})=\oplus\, \mathcal{C}_{\underline{d},b}(V^{\oplus 6})$,
where $\underline{d}=(d_1,\ldots,d_6)$ refers to the degree in the 
coefficients $l_{i,1},l_{i,2}$ and $b$ to the degree in $x_1,x_2$.

The dimension of $\mathcal{C}_{d,b}(V^{\oplus 6})$ for $d_i=d$ for $i=1,\ldots,6$ 
can also be found in \cite[p. 57]{Geyer} and it is given by
the coefficient of $z^{3d+b/2}$ in the Taylor expansion about $z=0$ of
\[
\Psi(z)=(1-z^{b+1})\left(\frac{1-z^{d+1}}{1-z}\right)^6\, ,
\]
so for $b$ odd we get $\dim \mathcal{C}_{d,b}(V^{\oplus 6})=0$.

Using the identification $\mathcal{C}(V^{\oplus 6})\cong \mathcal{I}(V^{\oplus 7})$,
and the fact that the $p_{ij}$ with $1 \leq i < j \leq 7$ generate this ring of
invariants of binary septics we find generators for $\mathcal{C}(V^{\oplus 6})$
by writing $l_7=l_{7,1}x_1+l_{7,2}x_2$ and substituting $l_{7,1}=-x_2$ and $l_{7,2}=x_1$
in invariants of binary septics.
For example, the generator of the $1$-dimensional space $\mathcal{C}_{1,6}(V^{\oplus 6})$, the
universal binary sextic, is given by
\[
C_{1,6}=p_{17}p_{27}p_{37}p_{47}p_{57}p_{67}
=l_{11} l_{21} l_{31} l_{41} l_{51} l_{61} x_1^{6}+\ldots+l_{12} l_{22} l_{32} l_{42} l_{52} l_{62} x_2^{6}\, .
\]

\bigskip

We will make use of a smaller ring
$\mathcal{C}^{\prime}(V^{\oplus 6})$  of covariants, the subring of 
$\mathcal{C}(V^{\oplus 6})$  where the covariants have
the same degree in the coefficients of the six linear forms $l_i$, see equation
\ref{samedegree} in the next section. 
The ring of invariants $\mathcal{I}'(V^{\oplus 6})$
is the coordinate ring (section ring) of the
 GIT quotient $({\PP}^1)^6//{\rm PGL}(2)$. This GIT quotient
 can be identified with the Segre cubic
given in ${\PP}^5$ with coordinates $y_1,\ldots,y_6$ by 
the equations
\begin{equation}\label{Segre}
\sigma_1=0, \, \sigma_3=0\, , 
\end{equation}
where $\sigma_i$ is the $i$th elementary symmetric function in the $y_1,\ldots,y_6$,
see \cite{D-O}.
The generating function of the dimension of $\mathcal{C}^{\prime}_{d,b}(V^{\oplus 6})$ 
can be computed from $\Psi(z)$; it is of the form $N/(1-t)^5$ with $N=N(s,t)$ a polynomial in $s^2$ of
degree $40$ in $s$ with as coefficient of $s^{2j}$ a polynomial in $t$ of degree $\leq j+2$.
\end{section}
\begin{section}{From modular forms to covariants and back}
The fact that the pull back of the Hodge bundle to the stack $[\mathcal{P}/{\rm GL}_2]$ under the isomorphism
 $[\mathcal{P}/{\rm GL}_2]\cong \mathcal{M}_2[2]$
is the equivariant bundle $V$
defines an injective homomorphism from modular forms to covariants
$$
\mu: \mathcal{R}(\mathcal{A}_2[2]) \to \mathcal{C}(V^{\oplus 6})
$$
that sends $M_{j,k}(\Gamma_2[2])$ to $\mathcal{C}^{\prime}_{k+j/2,j}(V^{\oplus 6})$. Here 
$$
M_{j,k}(\Gamma_2[2])=H^0(\widetilde{\mathcal{A}}_2[2], {\rm Sym}^j({\EE}) \otimes \det({\EE})^k)
$$ 
denotes the space of Siegel
modular forms of weight $(j,k)$ of level $2$ and $\mathcal{R}(\mathcal{A}_2[2])$ is the bigraded ring
$$
\mathcal{R}(\mathcal{A}_2[2])=\oplus_{j,k} M_{j,k}(\Gamma_2[2])\, .
$$
\begin{lemma} \label{degreeforscalar}
The image under $\mu$ of a modular form on $\mathcal{A}_2[2]$ lands in the subring $\mathcal{C}^{\prime}(V^{\oplus 6})$
 of   $\mathcal{C}(V^{\oplus 6})$ generated by covariants that have the same degree in the coefficients
of the six linear forms $l_i$.
This means that if we write $\ell_i=\ell_{i,1}x_1+\ell_{i,2}x_2$ then for a monomial
$\prod_i\ell_{i,1}^{\alpha_i}\ell_{i,2}^{\beta_i}$ occurring in a invariant of degree $d$ 
we have 
\begin{equation}\label{samedegree}
\alpha_i+\beta_i=d \quad \text{\rm for $i=1,\ldots,6$.}
\end{equation}
\end{lemma}
\begin{proof}
The image of the map $\mu$ lands in the subring of covariants that have the same degree
in the coordinates of all the six $l_i$ since we identify all six summands $V$ of $V^{\oplus 6}$
with the fibre of the Hodge bundle ${\EE}$.   
\end{proof}

We know that the ring of modular forms of even weight is generated by the fourth powers of the even theta characteristics. This ring can be identified with the coordinate ring
(section ring) of the Igusa quartic, see \cite{CvdGG,vdG}.
We have for $\vartheta^4_i=\vartheta_{\pi}^4$ 
corresponding to the partition $\pi=(abc)(def)$ that its image is given by
\begin{equation}\label{sixpij}
\mu(\vartheta_{\pi}^4)=p_{ab}p_{ac}p_{bc}p_{de}p_{df}p_{ef}
\end{equation}
The form $\chi_5$ corresponds to the product of the $15$ invariants $p_{ij}$. The
full ring of modular forms is generated by $\chi_5$ over the subring of even weight.

\begin{remark}\label{outerauto}
The action of $\mathfrak{S}_6$ on the $p_{ij}$ as in \ref{sixpij} determines an action on the space
$M_2(\Gamma_2[2])$ as $s[2^3]$. The fact that $\mathfrak{S}_6$ admits an outer automorphism
reconciles this with the literature where one finds it as $s[3^2]$. See the paper \cite{HMSV}
for more on this outer automorphism. 
\end{remark}

For the modular form $G_i^4={\rm Sym}^4(G_i)$ of weight $(4,2)$ we know that its divisor on ${\PP}({\EE})$ 
over $\mathcal{A}_2[2]$ coincides with the closure of $4\, W_i$, hence we get by
Proposition \ref{Wiclass} for $i=1,\ldots,6$
$$
\mu(G_i^4)= \ell_i^4 \prod_{a<b, a\neq i \neq b} p_{ab}\, ,
$$
up to a multiplicative constant.

\bigskip
Under the map $\mu$ the modular form $\vartheta_i^4=\vartheta_{\pi}^4$ 
corresponds to the product of six $p_{ij}$ given in \ref{sixpij} and $G_i^4$ to $l_i^4 \prod_{a\neq i\neq b} p_{ab}$.

A covariant in $\mathcal{C}^{\prime}(V^{\oplus 6})$ of bidegree $(d,b)$ can be written as
\begin{equation}\label{Pi}
C_{d,b}= \sum_{j=0}^b P_j \, x_1^{b-j}x_2^j
\end{equation}
with $P_j$ a polynomial in the coefficients $l_{i,1}$, $l_{i,2}$ of the six linear forms that has the same degree $d$
for $i=1,\ldots,6$ in the sense as specified in equation \ref{samedegree}.

\begin{definition}\label{define-nu}
In order to define a map $\nu: \mathcal{C}'(V^{\oplus 6}) \to \mathcal{R}(\mathcal{A}_2[2])[1/\chi_5]$ 
in a convenient way we define $\nu$ by the substitution in $C_{d,b}$ given on the generators $l_i$ by
$$
l_i \mapsto G_i/(\prod_{a<b, a \neq i \neq b} p_{ab})^{1/4}
$$
Then the coefficients $P_j$ 
give the coefficients of the vector-valued modular form. 
\end{definition}
The right hand side is
just a formal expression, but we know that by applying this to a covariant 
from $\mathcal{C}^{\prime}(V^{\oplus 6})$ we will land in the ring of 
meromorphic modular forms of level~$2$.

\begin{proposition}
By substituting $G_i/(\prod_{a<b, a \neq i \neq b} p_{ab})^{1/4}$ for $\ell_i$ in a covariant we get an injective homomorphism
$$
\nu: \mathcal{C}^{\prime}(V^{\oplus 6}) \xrightarrow{} \mathcal{R}(\mathcal{A}_2[2])[1/\chi_5]
$$
that sends an element of $\mathcal{C}_{d,b}^{\prime}(V^{\oplus 6})$ to a meromorphic form of weight
$(b,d-b/2)$ in $\mathcal{R}(\mathcal{A}_2[2])[1/\chi_5]$.
The image is holomorphic on $\mathcal{M}_2[2]$.
\end{proposition}
\begin{proof} We know that the image lands in level $2$.
Then the statement follows immediately from Proposition 4.2 of \cite{CvdG2023}.
\end{proof}

We thus get an injective homomorphisms
$$
\mathcal{R}(\mathcal{A}_2[2]) \langepijl{\mu} \mathcal{C}^{\prime}(V^{\oplus 6})\langepijl{\nu} \mathcal{R}(\mathcal{M}_2[2])
$$
with $\nu\circ \mu$ the identity on $\mathcal{R}(\mathcal{A}_2[2]) \subset \mathcal{R}(\mathcal{M}_2[2])$.

\end{section}
\begin{section}{A criterion for regularity}

The homomorphism
 $\nu:C'(V^{\oplus6})\longrightarrow
 R(\mathcal M_2[2])=R(\mathcal A_2[2])[1/\chi_5]$
produces meromorphic modular forms whose possible poles on
$\mathcal A_2[2]$ are supported on the ten irreducible components
$H_\pi$ of $\mathcal A_{1,1}[2]$.  We give a way to calculate the order
along these divisors.

Let
\[
 \pi=S\sqcup S^c,\qquad \# S= \#S^c=3,
\]
be a partition of $\{1,\ldots,6\}$ into two triples.  Introduce
independent variables $a_i,b_i$ and make the simultaneous
substitution
\begin{equation}\label{eq:simultaneous-substitution}
 l_i=l_{i,1}x_1+l_{i,2}x_2\longmapsto
 \begin{cases}
   a_i x_1+t b_i x_2,&i\in S,\\
   t a_i x_1+b_i x_2,&i\in S^c.
 \end{cases}
\end{equation}
If $C\in C'_{d,b}(V^{\oplus6})$, the result is a polynomial in
$t,a_i,b_i,x_1,x_2$.

\begin{definition}\label{def:wpi}
We put
\[
 w_\pi(C)=
 {\ord}_{t=0}\psi_{\pi,t}(C),
\]
where $\psi_{\pi,t}$ denotes the substitution
\eqref{eq:simultaneous-substitution}, and define
\begin{equation}\label{eq:vpi}
 v_\pi(C)=w_\pi(C)-d.
\end{equation}
\end{definition}

Interchanging $S$ and $S^c$, together with interchanging $x_1$ and
$x_2$, does not change the order.  Hence $w_\pi$ depends only on the
unordered partition $\pi$.

This procedure induces also a filtration on the Pl\"ucker coordinates.
Recall that
\[
 p_{ij}=l_{i,1}l_{j,2}-l_{i,2}l_{j,1}.
\]
Under \eqref{eq:simultaneous-substitution} one obtains
\begin{equation}\label{eq:pij-orders}
 w_\pi(p_{ij})=
 \begin{cases}
  1,&i,j\in S\ \text{or}\ i,j\in S^c,\\
  0,&i\in S,\ j\in S^c\ \text{or conversely}.
 \end{cases}
\end{equation}
Indeed, for $i,j\in S$,
one has $
 \psi_{\pi,t}(p_{ij})
 =t(a_i b_j-b_i a_j)$, 
while for $i\in S$ and $j\in S^c$,
one has $
 \psi_{\pi,t}(p_{ij})
 =a_i b_j-t^2b_i a_j$.

It is useful to introduce notation for the leading terms.  For a
Pl\"ucker coordinate define
\begin{equation}\label{eq:initial-pij}
 \overline p_{ij}
 =
 \left.
 t^{-w_\pi(p_{ij})}\psi_{\pi,t}(p_{ij})
 \right|_{t=0}.
\end{equation}
so the bar means ``divide by the smallest
occurring power of $t$ and then put $t=0$'', and not simply ``put
$t=0$''.
Thus
\[
 \overline p_{ij}=
 \begin{cases}
  a_i b_j-b_i a_j,&i,j\in S,\\
  a_i b_j-b_i a_j,&i,j\in S^c,\\
  a_i b_j,&i\in S,\ j\in S^c,
 \end{cases}
\]
For example, if $i,j\in S$ and $k,l\in S^c$, the Pl\"ucker relation
\[
 p_{ik}p_{jl}-p_{il}p_{jk}=p_{ij}p_{kl}
\]
has terms of $t$-orders $0,0,2$.  Its lowest-order
relation is therefore
$
\overline p_{ik}\,\overline p_{jl}
 - \overline p_{il}\,\overline p_{jk}=0$.
Equivalently, the $3\times3$ matrix
$(\overline p_{ij})_{i\in S,j\in S^c}$ has rank one.

We now make the comparison of Definition~\ref{def:wpi} 
with the valuation defined by the divisor $H_\pi$.
We use the quotient description
of $\mathcal M_2[2]$ by six ordered points of $\mathbf P^1$ from
Section~2.

Let $\eta_\pi$ be the generic point of $H_\pi$.  The corresponding
stable marked rational curve consists of two copies of $\mathbf P^1$
meeting in one point, with the three markings indexed by $S$ on one
component and the three markings indexed by $S^c$ on the other.
Let $t$ be a local equation for $H_\pi$ at $\eta_\pi$.

\begin{lemma}\label{lem:local-3plus3}
After passing to an
\'etale neighbourhood and choosing homogeneous coordinates on the
hyperelliptic quotient, the six branch points can be represented by
vectors
\begin{equation}\label{eq:local-six-points}
 u_i(t)=
 \begin{cases}
  (a_i,tb_i),&i\in S,\\
  (ta_i,b_i),&i\in S^c,
 \end{cases}
\end{equation}
where the $a_i,b_i$ are units at the generic point.  The parameters
describing the positions of the three points on each component are
generic.
\end{lemma}

\begin{proof}
Let the stable genus-$2$ curve at a general point of $H_\pi$ be
$E_1\cup E_2$, where the two elliptic components meet in a node.
The hyperelliptic involution restricts to the elliptic involution on
each component, and the node is a ramification point of each of the
double covers $E_i \to {\PP}^1$.
Choose local parameters $x,y$ on the two branches of the genus-$2$
curve at the node.  A smoothing of the node is
$ xy=t$.
After multiplying $x$ and $y$ by units, the hyperelliptic involution
acts locally by $(x,y) \mapsto (-x,-y)$ and
invariant parameters on the two branches of the rational
quotient are $u=x^2$, $y=v^2$
and $uv=t^2$.
Thus the smoothing parameter of the node on the stable six-pointed
rational quotient is $t^2$.  This is the familiar ramification of the
map from the stack of admissible double covers to the boundary of
$\overline{\mathcal M}_{0,6}$ at a node which is ramified in the
double cover.

Let the three branch points indexed by $S$ have, on the $u$-component
of the special fibre, non-zero coordinates $\alpha_i$, and let the
three branch points indexed by $S^c$ have, on the $v$-component,
non-zero coordinates $\beta_j$.  On a nearby smooth fibre of
$uv=t^2$, using $u$ as affine coordinate, their
coordinates have the form
\[
 u_i=\alpha_i+O(t^2)x\quad \text{for $ i\in S$},
\quad \text{and} \quad
 u_j={t^2}/{(\beta_j+O(t^2))}\quad \text{for $j\in S^c$}.
\]
With $z=u/t$ we find
\[
 z_i=\frac{\alpha_i}{t}+O(t)\quad \text{for $i\in S$},
 \qquad
 z_j=\frac{t}{\beta_j}+O(t^3)\quad \text{for $j\in S^c$}.
\]
We thus find points of the form
\[
 (a_i:tb_i)\quad \text{for $i\in S$},
 \qquad
 (ta_j:b_j)\quad \text{for $j\in S^c$},
\]
with $a_i,b_j$ units.  The remaining coefficients record the positions
of the three marked branch points on each of the two components.

At the generic point of $H_\pi$ these two ordered four-pointed
rational curves (the three branch points together with the node) are
generic.  Together with the transverse
parameter $t$, they give the three local moduli parameters.  This
proves the assertion.
\end{proof}

Lemma~\ref{lem:local-3plus3} gives precisely the substitution
\eqref{eq:simultaneous-substitution}.  Notice that $t$ is 
the transverse parameter on the moduli stack of genus-$2$
curves (and hence, via the Torelli map, at a general point of
$H_\pi$), while $t^2$ is the transverse parameter on the
coarse stable six-pointed rational quotient.

Equivalently, Definition~\ref{def:wpi} is the divisorial valuation
on the covariant ring associated with the separating divisor
$H_\pi$ before the scalar normalization occurring in
Definition~\ref{define-nu}.

We now compare the vectors $u_i(t)$ with the theta-gradient
description used in Definition~\ref{define-nu}.  At a general point of $H_\pi$,
choose a local frame $e_1,e_2$ of the Hodge bundle adapted to the two
elliptic factors, and use the same transverse parameter $t$.  The
gradients of the six odd theta functions have the form
\begin{equation}\label{eq:gradient-expansion}
 \begin{aligned}
  G_i&=A_i e_1+tB_i e_2+O(t^2),&&i\in S,\\
  G_j&=t\, C_j e_1+D_j e_2+O(t^2),&&j\in S^c,
 \end{aligned}
\end{equation}
where $A_i$ and $D_j$ are non-zero at the generic point.  In
particular,
\begin{equation}\label{eq:wedge-orders}
 \epsilon_{ij}:={\ord}_{H_\pi}(G_i\wedge G_j)=
 \begin{cases}
  1,&i,j\in S\ \text{or}\ i,j\in S^c,\\
  0,&i\in S,\ j\in S^c\ \text{or conversely}.
 \end{cases}
\end{equation}
The projective points defined by the $G_i$ are the six marked
Weierstrass points.  Therefore Lemma~\ref{lem:local-3plus3}, applied
to these projective points, gives
\begin{equation}\label{eq:C-of-G-order}
 {\ord}_{H_\pi}C(G_1,\ldots,G_6)=w_\pi(C)
\end{equation}
after removing common invertible scalar factors. 

It remains to take account of the non-invertible scalar normalization
in Definition~\ref{define-nu}.  Write the vector substituted for $l_i$ as
\[
 L_i=s_iG_i,\qquad x_i={\ord}_{H_\pi}(s_i).
\]
After the substitution, a bracket occurring in the denominator of
Definition~\ref{define-nu} is
\[
 p_{ab}(L)=L_a\wedge L_b
          =s_as_b(G_a\wedge G_b),
\]
and hence
\[
 {\ord}_{H_\pi}p_{ab}(L)=x_a+x_b+\epsilon_{ab}.
\]
It follows from Definition~\ref{define-nu} that
\begin{equation}\label{eq:xi-equations}
 x_i=-\frac14
 \sum_{\substack{a<b\\a,b\neq i}}
 \bigl(\epsilon_{ab}+x_a+x_b\bigr).
\end{equation}
By \eqref{eq:wedge-orders}, precisely six of the fifteen numbers
$\epsilon_{ab}$ are equal to $1$, namely the three pairs contained in
$S$ and the three pairs contained in $S^c$.  Thus
 $\sum_{a<b}\epsilon_{ab}=6$.
Summing \eqref{eq:xi-equations} over $i=1,\ldots,6$, each
$\epsilon_{ab}$ occurs four times and each $x_i$ occurs twenty times
inside the sums.  Therefore, with $X=\sum_i x_i$,
 $X=-\frac14(4\cdot6+20X)=-6-5X$,
and consequently $ \sum_{i=1}^6x_i=-1$.

Now $C\in C'_{d,b}(V^{\oplus6})$ is homogeneous of degree $d$ in
the coefficients of each of the six linear forms and we can write
\[
 C(L_1,\ldots,L_6)
 =
 \left(\prod_{i=1}^6s_i^d\right)
 C(G_1,\ldots,G_6)\, .
\]
Together with $\sum x_i=-1$ and \eqref{eq:C-of-G-order} this gives
 ${\ord}_{H_\pi}\nu(C)=w_\pi(C)-d$ and we obtain the
following proposition.

\begin{proposition}\label{prop:valuation}
Let $\pi$ be a partition of $\{1,\ldots,6\}$ into two triples and let
$C\in C'_{d,b}(V^{\oplus6})$.  Then
 $$
{\ord}_{H_\pi}\nu(C)=v_\pi(C)=w_\pi(C)-d\, .
 $$
\end{proposition}
As the divisor $\mathcal A_{1,1}[2]$ is the union of the ten irreducible
components $H_\pi$ this implies the following.

\begin{criterion}\label{crit:regularity}
Let $C\in C'_{d,b}(V^{\oplus6})$.  Then $\nu(C)$ is holomorphic along
$\mathcal A_{1,1}[2]$ if and only if
 $w_\pi(C)\geq d$
for every partition $\pi$ of $\{1,\ldots,6\}$ into two triples.
Equivalently,
$v_\pi(C)\geq 0$
for all ten partitions.
\end{criterion}

\begin{example}\label{ex:universal-sextic}
Consider the universal binary sextic
 $C_{1,6}=l_1l_2l_3l_4l_5l_6$.
For every partition $\pi$, the simultaneous substitution has a
non-zero term of order zero.  Therefore
$w_\pi(C_{1,6})=0$ and $v_\pi(C_{1,6})=-1$
and $\nu(C_{1,6})$ has a simple pole along each $H_\pi$.
On the other hand,
\[
 I_5=\prod_{1\leq i<j\leq6}p_{ij}.
\]
For every $3+3$ partition there are six pairs contained in one of the
two triples, and hence
 $w_\pi(I_5)=6$.
Since $I_5$ has degree $5$ in each $l_i$,
$v_\pi(I_5)=6-5=1$.
It follows that
$v_\pi(I_5C_{1,6})=0$
for every $\pi$, in agreement with the holomorphicity of the modular form
with character 
$\nu(I_5C_{1,6})=\chi_{6,3}$.
\end{example}

\end{section}
\begin{section}{Constructing vector-valued modular forms from covariants}
As shown above we can construct the space $M_{j,k}(\Gamma_2[2])$ as the
image under $\nu$ of the subspace of $\mathcal{C}_{k+j/2,j}^{\prime}(V^{\oplus 6})$
of covariants $C$ that satisfy the holomorphicity criterion $v_{\pi}(C)\geq 0$
for all ten $\pi$. For the dimensions of the spaces of modular forms we refer to \cite{BC}.

The spaces $\mathcal{C}_{d,b}^{\prime}(V^{\oplus 6})$ tend to have very large
dimension, so in practice we start with covariants of bidegree $(d,b)$ with smaller
$b$ and then get rid of the poles by multiplying with a power of $\chi_5$.
We note that $\chi_5= \nu(I_5)$ with $I_5=\prod p_{ij}$, the square root of the
discriminant. Indeed, we have $v_{\pi}(I_5)=1$ for all $\pi$.

The ring $\mathcal{I}(V^{\oplus 6})$ is well-known as we saw above. As to
the action of $\mathfrak{S}_6$ we have
\[
\mathcal{I}_d(V^{\oplus 6})\cong \Sym^d(s[3^2])-\
\begin{cases}
0 \quad \hspace{3.1cm}\text{if} \quad 0\leqslant d\leqslant 2\\
\Sym^{d-3}(s[3^2])\otimes s[1^6] \quad \text{if} \quad d\geqslant 3\\
\end{cases}
\]
with the convention $\Sym^0(s[3^2])=s[6]$.
For the ring of covariants $\mathcal{C}^{\prime}(V^{\oplus 6})$ we do not
have a closed formula for the decomposition of $\mathcal{C}_{d,b}^{\prime}(V^{\oplus 6})$
as an $\mathfrak{S}_6$-representation. We construct covariants, apply the action of
$\mathfrak{S}_6$ and try to generate a basis.
Let us begin with a simple example. 
\begin{example}
The space $\mathcal{C}^{\prime}_{1,2}(V^{\oplus 6})$
is $s[4,2]$ as $\mathfrak{S}_6$-representation and is generated by 
\[
C_0= p_{36}p_{45}\left(l_{1,1} l_{2,1}\,x_1^2+ \left(l_{1,1} l_{2,2}+l_{1,2} l_{2,1}\right)\,x_1x_2+l_{1,2} l_{2,2} \,x_2^2\right).
\]
One checks that $\nu(I_5C_0)\in S_{2,5}(\Gamma_2[2])$ and generates it as an 
$\mathfrak{S}_6$-representation $s[2^2,1^2]$
with $\dim S_{2,5}(\Gamma_2[2])=9$. In \cite[Section 15 and 20]{CvdGG}, generators of the 
space $S_{2,5}(\Gamma[2])$ have been constructed by using so-called Rankin-Cohen brackets:
$\Phi_i=[\vartheta_i^4,\chi_5]/\chi_5$ with the relation $\sum_{i=1}^{10}\Phi_i=0$.
We  can express one set of generators in the other one.
\end{example}

\bigskip
As an illustration we now use covariants of bidegree $(d,b)$  with $b=4$ and $1 \leq d \leq 2$
to construct vector-valued modular forms. For simplicity we just write $\mathcal{C}^{\prime}_{d,b}$
for $\mathcal{C}^{\prime}_{d,b}(V^{\oplus 6})$.

{\bf Case $(d,b)=(1,4)$.} 
We have $\mathcal{C}^{\prime}_{1,4}=s[5,1]$ as $\mathfrak{S}_6$-representation. It can be generated by the
covariant
$$
C= p_{12}l_3l_4l_5l_6
$$
and one can check that $v_{\pi}(C)\geq -1$ for all $\pi$. This yields that $\nu(I_5C) \in M_{4,4}(\Gamma_2[2])$,
but as we know
$$
M_{4,4}(\Gamma_2[2])=s[4,2]\oplus s[3,2,1]\oplus s[2^3]\oplus s[2,1^4], \quad S_{4,4}(\Gamma_2[2])=s[2,1^4]
$$
it follows that $\nu(I_5C)$ generates $S_{4,4}(\Gamma_2[2])$.
We can identify this form up to a multiplicative constant with 
$\vartheta_7\vartheta_8\vartheta_9\vartheta_{10} {\rm Sym}^4(G_3,G_4,G_5,G_6)$.

\bigskip
{\bf Case $(d,b)=(2,4)$.} Here $\dim \mathcal{C}^{\prime}_{2,4}=40$ and we show that
$$
\mathcal{C}^{\prime}_{2,4}(V^{\oplus 6})= s[6]\oplus s[5,1]\oplus 2\, s[4,2] \oplus s[3,2,1] \, .
$$
One piece of this can be generated by taking $\mathcal{I}_1\otimes \mathcal{C}^{\prime}_{1,4}$
which generates $s[3^2]\otimes s[5,1]=s[4,2]\oplus s[3,2,1]$. The isotypic
component $s[3,2,1]$ is generated by $H=p_{12}p_{45}p_{46}p_{56}l_1l_2l_3^2$ and
one verifies $v_{\pi}(H)\geq -1$ for all $\pi$. We conclude that $\nu(I_5H)$ generates the $s[3,2,1]$
component of 
$$
S_{4,5}(\Gamma_2[2])=s[3,2,1]\oplus s[2^2,1^2]\oplus s[2,1^4] \, .
$$
To generate the remaining isotypic components we consider the $36$ covariants 
$$
C_{ij}=l_i l_j \left(\frac{l_1l_2l_3l_4l_5l_6}{l_i},\frac{l_1l_2l_3l_4l_5l_6}{l_j}\right)_4
\qquad 1 \leq i,j\leq 6\, ,
$$
where the lower index $4$ indicates the $4$th transvectant  (\"Uberschiebung). 
We refer to \cite[Section 2]{CFG2}
for more on the theory of covariants.
These covariants generate a $15$-dimensional space $s[6]\oplus s[5,1]\oplus s[4,2]$.
The $s[6]$-part of $\mathcal{C}^{\prime}_{2,4}$ is generated by the covariant
$$
75 (f_6,f_6)_4= (10a_0a_4-5a_1a_3+2a_2^2)x_1^4+ \cdots
$$
with $f_6=a_0x_1^6+\ldots + a_6x_2^6$ the universal binary sextic.
Thus we have all of $\mathcal{C}^{\prime}_{2,4}$. We now show that we can generate the whole
space $S_{4,5}(\Gamma_2[2])$ with these covariants. 
The projection on the isotypic component $s[5,1]\subset \mathcal{C}^{\prime}_{2,4}$
of any $C_{ij}$ has valuation $v_{\pi}\geq -1$ and we thus can generate
the $s[2,1^4]$-part of $S_{4,5}(\Gamma_2[2])$ by $\nu(I_5[s[5,1])$. 
To obtain the remaining part $s[2^2,1^2]$ of $S_{4,5}(\Gamma_2[2])$ we let $W$
be the orbit under $\mathfrak{S}_6$ of the projection of the covariant $C_{12}$ 
on the isotypic component $2s[4,2] \subset \mathcal{C}^{\prime}_{2,4}$.
One checks that any element of $W$ has valuation $v_{\pi}\geq -1$ for all $\pi$.
Thus we can generate the $s[2^2,1^2]$ component of $S_{4,5}(\Gamma_2[2])$ as
$\nu(I_5 \, W)$.

\bigskip
We further explain and illustrate the method by working with covariants of bidegree $(d,6)$
with $1 \leq d \leq 6$.

\noindent
{\bf Case $d=1$.} The space 
$\mathcal{C}^{\prime}_{1,6}$ is generated by the obviously $\mathfrak{S}_6$-invariant 
covariant $C_{1,6}=\ell_1\cdots \ell_6$,
see Example \ref{ex:universal-sextic}. 
So $\nu(I_5C_{1,6})$ is a holomorphic modular form  
of weight $(6,3)$ on $\Gamma_2[2]$ and in fact
$$
\chi_{6,3}= {\rm Sym}^6(G_1,\ldots,G_6) \, .
$$
{\bf Case $d=2$.} 
We have $15$ invariants of degree $1$, namely $p_{ab}p_{cd}p_{ef}$, where
$\{a,b,c,d,e,f\}=\{1,2,3,4,5,6\}$, but by the Pl\"ucker relations only five
are linearly independent. We choose as basis of $\mathcal{I}_1(V^{\oplus 6})$:
\[
i_1=p_{12}p_{34}p_{56}, 
i_2=p_{12}p_{35}p_{46}, 
i_3=p_{13}p_{24}p_{56}, 
i_4=p_{13}p_{25}p_{46}, 
i_5=p_{14}p_{25}p_{36}
\]
The space of covariants of bidegree $(2,6)$ decomposes as
\[
\mathcal{C}'_{2,6}(V^{\oplus 6})=s[5,1]\oplus s[4,2]\oplus s[4,1^2] \oplus s[3^2].
\]
Indeed, we can write a basis of $\mathcal{C}_{2,6}^{\prime}(V^{\oplus 6})$  by starting with
$C_{j} = C_{1,6} i_{j}$ for $j=1,\ldots,5$ 
and continuing with
\begin{align*}
&C_{6} = l_{1}^{2} l_{2}^{2} l_{3}^{2} p_{45} p_{46} p_{56}, 
C_{7} = l_{1}^{2} l_{2}^{2} l_{3} l_{4} p_{36} p_{45} p_{56}, 
C_{8} = l_{1}^{2} l_{2}^{2} l_{3} l_{5} p_{36} p_{45} p_{46}, 
C_{9} = l_{1}^{2} l_{2}^{2} l_{4} l_{5} p_{36}^{2} p_{45}, \\
&C_{10} = l_{1}^{2} l_{2}^{2} l_{3} l_{4} p_{35} p_{46} p_{56}, 
C_{11} = l_{1}^{2} l_{2}^{2} l_{4}^{2} p_{35} p_{36} p_{56}, 
C_{12} = l_{1}^{2} l_{3} l_{4}^{2} l_{5} p_{26}^{2} p_{35}, 
C_{13} = l_{1}^{2} l_{2} l_{4}^{2} l_{6} p_{25} p_{35} p_{36}, \\
&C_{14} = l_{1}^{2} l_{3} l_{4}^{2} l_{5} p_{25} p_{26} p_{36}, 
C_{15} = l_{1}^{2} l_{2} l_{3} l_{5}^{2} p_{24} p_{36} p_{46}, 
C_{16} = l_{1}^{2} l_{2} l_{3} l_{5} l_{6} p_{24} p_{36} p_{45}, 
C_{17} = l_{1}^{2} l_{2} l_{5} l_{6}^{2} p_{24} p_{34} p_{35}, \\
&C_{18} = l_{1}^{2} l_{2} l_{3} l_{5} l_{6} p_{23} p_{45} p_{46}, 
C_{19} = l_{1}^{2} l_{2} l_{3} l_{6}^{2} p_{23} p_{45}^{2}, 
C_{20} = l_{1}^{2} l_{3} l_{4} l_{5}^{2} p_{23} p_{26} p_{46}, 
C_{21} = l_{1} l_{2}^{2} l_{4}^{2} l_{6} p_{16} p_{35}^{2}, \\
&C_{22} = l_{1} l_{2} l_{3}^{2} l_{4} l_{6} p_{16} p_{25} p_{45}, 
C_{23} = l_{1} l_{2} l_{3}^{2} l_{5} l_{6} p_{16} p_{24} p_{45}, 
C_{24} = l_{2} l_{3}^{2} l_{4}^{2} l_{5} p_{16}^{2} p_{25}, 
C_{25} = l_{2} l_{3}^{2} l_{4} l_{5}^{2} p_{16}^{2} p_{24}, \\
&C_{26} = l_{2} l_{3} l_{4}^{2} l_{5}^{2} p_{16}^{2} p_{23}, 
C_{27} = l_{1} l_{2}^{2} l_{3}^{2} l_{4} p_{15} p_{46} p_{56}, 
C_{28} = l_{3} l_{4}^{2} l_{5}^{2} l_{6} p_{12}^{2} p_{36}, 
C_{29} = l_{3} l_{4} l_{5}^{2} l_{6}^{2} p_{12} p_{14} p_{23}
\end{align*}

The $s[3^2]$-component 
 is generated by the $\mathfrak{S}_6$-orbit of $C_1$.
The $s[4,1^2]$-component is generated by the orbit of $C_6$ or $C_{11}$.
The $s[5,1]$-component can be generated by
$$
4C_1-4C_2-C_3+C_4+6C_5-7C_7+4C_8+C_9+C_{10}+3C_{11}+C_{13}-C_{15}+C_{21}-C_{22}-C_{27}+5C_{28}
$$
and a generator of the $9$-dimensional piece $s[4,2]$ is 
$$
2C_1-2C_2+C_7-3C_8+C_9-4C_{16}\, .
$$
Since each element of $\mathcal{C}'_{2,6}(V^{\oplus 6})$ produces a 
meromorphic modular form of weight $(6,-1)$ on $\Gamma[2]$ we are 
looking for elements $C$ in 
$\mathcal{C}'_{2,6}(V^{\oplus 6})$ such that
$
w_{\pi}(C)\geqslant -1 $ 
for any partition $\pi$ (equivalently $v_{\pi}(C)\geq -1$).
Then multiplying by $\chi_5\in S_{0,5}(\Gamma_2[2])\cong s[1^6]$ 
we find modular forms in $M_{6,4}(\Gamma_2[2])$.
We know the isotypical decomposition of the spaces 
$M_{6,4}(\Gamma_2[2])$ and $S_{6,4}(\Gamma_2[2])$; they are given by

\begin{align*}
M_{6,4}(\Gamma_2[2]) &\cong s[4, 2] \oplus  s[3, 2, 1] \oplus  2\,s[3, 1^3] \oplus  s[2^3] \oplus  s[2^2,1^2] \oplus  s[2,1^4];\\
S_{6,4}(\Gamma_2[2]) &\cong s[3, 1^3] \oplus  s[2^2,1^2]\, . 
\end{align*}
We check that no element in the $s[3^2]$ and $s[5,1]$-isotypic components satisfies the
condition on the valuations $v_{\pi}$. The elements in the $s[4,2]$ and $s[4,1^2]$-isotypic components do. Since we multiply with $\chi_5$ we thus find elements
$$
F \in \nu(s[4,2] I_5) \subset s[2^2,1^2] \subset M_{6,4}(\Gamma_2[2]), \quad
G \in  \nu(s[4,1^2] I_5)\subset s[3,1^3] \subset M_{6,4}(\Gamma_2[2])\, .
$$
\begin{lemma}
We have $F\in S_{6,4}(\Gamma_2[2])$ and $G \in S_{6,4}(\Gamma_2[2])$.
\end{lemma}
\begin{proof} For $F$ this is clear as the isotypic component $s[2^2,1^2]$ occurs with multiplicity $1$ in both $M_{6,4}(\Gamma_2[2])$
and $S_{6,4}(\Gamma_2[2])$. 
We know that $M_{6,4}(\Gamma_2[2])$ is the direct sum of the space of cusp form and the space 
${\rm KE}_{6,4}(\Gamma_2[2])$ of Klingen-Eisenstein series.
The irreducible representation $s[3,1^3]\subset KE_{6,4}(\Gamma_2[2])$ comes from the space 
of new forms of weight $10$ on the congruence subgroup $\Gamma_0(4)\subset {\rm SL}(2,\ZZ)$; 
this space is one-dimensional and generated by a form $f$ whose Fourier expansion starts with
$f(\tau)=q + 228q^3 - 666q^5\ldots $  with $\tau\in \mathfrak{H}_1$ and $q=e^{2\pi i \tau}$.
By applying the map $\nu$ we can compute the Fourier expansion of $G$ 
and we get the following Fourier coefficients 
\begin{align*}
& a_{G}(1,1,1) =(0,0,1,2,1,0,0)^t,\quad   a_{G}(3,3,3)=(-36, -108, 3, 186, 3, -108, -36)^t,\\
& a_{G}(7,9,3)=(36, 156, 277, 258, 133, 36, 4)^t, \quad a_{G}(5,5,5)=(0, 0, 10332, 20664, 10332, 0, 0)^t.
\end{align*}
The Hecke eigenvalues are $-24$ at $p=3$ and $10332$ at $p=5$ for $G$. But we have  $-24\neq 228(1+3^2)$  
(and/or $10332\neq -666(1+5^2)$), so that the forms $G_i$ for $0\leqslant i \leqslant 10$ generate the irreducible representation $s[3,1^3]\subset S_{6,4}(\Gamma[2])$.
\end{proof}

Then we can obtain the Fourier expansions for $F$. 
We can calculate the Hecke eigenvalues for $F$ at $p=3$ and $p=5$ and find
$-280$ and $-2980$ in agreement with the values given by the website \cite{website}
where the eigenvalues are obtained by point counting over finite fields.

\bigskip

As another example we take $\mathcal{C}'_{2,8}(V^{\oplus 6})$ which is $15$-dimensional and generated by
the orbit under $\mathfrak{S}_6$ of the following covariant
\[
l^2_1l^2_2l^2_3(l_4l_5l_6,l_4l_5l_6)_2
\]
Under the map $\nu$, such a covariant produces a necessarily meromorphic modular form of weight $(8,-2)$ on $\Gamma_2[2]$.
The space $W$ generated by the orbit  under $\mathfrak{S}_6$ of the covariant
\[
l^2_1l^2_2l^2_3(l_4l_5l_6,l_4l_5l_6)_2-l^2_4l^2_5l^2_6(l_1l_2l_3,l_1l_2l_3)_2
\]
is $5$-dimensional ($\cong s[5,1]$) and any element $C$ in this orbit satisfies $v_{\pi}(C)\geqslant -1$; 
therefore we have
\[
\nu(I_5W)\subset S_{8,3}(\Gamma_2[2])\cong s[2,1^4]
\]
and we thus  built the full space $S_{8,3}(\Gamma_2[2])$.
\end{section}

\begin{section}{Intermediate levels}
Using the Galois action of $\mathfrak{S}_6$ on $\mathcal{A}_2[2]$ we can describe all 
modular forms for the intermediate levels between level $1$ and $2$ using covariants
for the actions on a totally split binary sextic, that is, for the action of 
${\rm GL}_2$ on $V^{\oplus 6}$. But for the intermediate levels we can also use
different splittings of the universal binary sextic $f_6$.
We illustrate this by two examples, one for $\Gamma_0[2]$ and one for $\Gamma[w]$
by using the splitting of $f_6$ as a product of three quadratic forms or
as a product of a quintic form and a linear form.

For $\Gamma_0[2]$ the corresponding moduli space can be described as the quotient
stack of $\mathcal{A}_2[2]$ under $\mathfrak{S}_4 \times \mathfrak{S}_2\cong \Gamma_0[2]/\Gamma_2[2]$.
We know that the dimension of the space $M_{j,k}(\Gamma_0[2])$ is given by
\[
\dim M_{j,k}(\Gamma_0[2])=m_{s[6]}+m_{s[4,2]}+m_{s[2^3]}
\]
with $m_{s[\varpi]}$ the multiplicity of the irreducible representation $s[\varpi]$ in $M_{j,k}(\Gamma_2[2])$,
see \cite[Section 9]{CvdGG}.
The generating series
for the dimension of spaces of scalar-valued modular forms on $\Gamma_0[2]$ is given by
\[
\sum_{k \geqslant 0}\dim M_{k}(\Gamma_0[2])\,t^k
=\frac{1+t^{19}}{(1-t^2)(1-t^4)^2(1-t^6)}=1+t^{2}+3 t^{4}+4 t^{6}+7 t^{8}+9 t^{10}+14 t^{12}+\dots
\]
(see for example \cite[Remark 3.3]{BC} and the references therein). Our example
deals with $M_{2}(\Gamma_0[2])$. Of course,
the generator of the $1$-dimensional space
 $M_{2}(\Gamma_0[2])$ can be obtained by symmetrising any element in $M_{2}(\Gamma_2[2])$ 
with respect to $\Gamma_0[2]/\Gamma_2[2]\cong \mathfrak{S}_4\times \ZZ/2\ZZ$,
but we can also use invariants of $V_2^{\oplus 3}$. Here
 $V_2$ denotes the space of binary quadric forms. We can use it since the irreducible representation $s[2^3]$
appears in the dimension formula.  We denote the three 
general binary quadric forms by
 \[
 q_1=a_0x_1^2+a_1x_1x_2+a_2x_2^2, \quad  q_2=b_0x_1^2+b_1x_1x_2+b_2x_2^2 \quad \text{and} \quad  q_3=c_0x_1^2+c_1x_1x_2+c_2x_2^2.
 \]
The generating series of the dimension of the space of invariants of degree $(d_1,d_2,d_3)$ 
in the coefficients of $q_1$, $q_2$ and $q_3$ is given by
 \[
 \sum_{d_1,d_2,d_3\geqslant 0} \dim \mathcal{I}_{(d_1,d_2,d_3)}(V_2^{\oplus 3})t^{d_1}
u^{d_2}v^{d_3}= \frac{1+t u v }{\left(1-t^{2}\right) \left(1-u^{2}\right) \left(1-v^{2}\right) \left(1-u v \right) \left(1-t u \right) \left(1-v t \right)}
 \]
so we have $\dim (\mathcal{I}_{(2,2,2)}(V_2^{\oplus 3}))=5$; remark that we only consider invariants 
of the same degree in the coefficients of $q_1$, $q_2$ and $q_3$. We put
\[
t_{ij}=-2(q_i,q_j)_2 \quad \text{for} \quad 1 \leqslant i\leqslant j\leqslant 3. 
\]
Note that $t_{ii}$ is simply the discriminant of $q_i$.
A basis of the space $\mathcal{I}_{(2,2,2)}(V_2^{\oplus 3})$ is then given by
\[
\mathcal{I}_{(2,2,2)}(V_2^{\oplus 3})=
\left\{
t_{11}t_{23}^2, t_{12}^2t_{33}, t_{13}^2t_{22}, t_{11}t_{22}t_{33}, t_{12}t_{13}t_{23}
\right\} \, .
\]
We then write $q_1=l_1l_2$, $q_2=l_3l_4$ and $q_3=l_5l_6$ and apply the criterion for regularity to a linear combination of the elements in the previous basis
of $\mathcal{I}_{(2,2,2)}(V_2^{\oplus 3})$ to get the following element
\[
I_{2,2,2}=t_{11}t_{22}t_{33}-t_{12}t_{13}t_{23}
\]
which satisfies $v_{\pi}(I_{2,2,2})\geqslant 0$ for any partition $\pi$. We then check that
\[
I_{2,2,2}=-2([1,4,6]+[1,3,6]+[1,3,5]+[1,4,5]) \quad \text{with} \quad [a,b,c]= p_{ab}p_{ac}p_{bc}p_{de}p_{df}p_{ef}
\] 
and therefore the image $\nu(I_{2,2,2})$ is, up to a multiplicative constant,
$\vartheta_1^4+\vartheta_2^4+\vartheta_3^4+\vartheta_4^4$, the generator of $M_2(\Gamma_0[2])$.

Note that any splitting of $\chi_{6,-2}$ into the product of $3$ meromorphic modular forms of 
weight $(2,k_i)$ with $k_1+k_2+k_3=-2$ can be used as a substitute for the substitution $\nu$. For example, we can choose
\[
\chi_{2,-2}=\frac{\Sym^2(G_1,G_2)}{\vartheta_1\ldots\vartheta_6}, \quad \chi^{(1)}_{2,0}=\frac{\Sym^2(G_1,G_2)}{\vartheta_7\vartheta_8} \quad \text{and} 
\quad 
\chi^{(2)}_{2,0}=\frac{\Sym^2(G_5,G_6)}{\vartheta_9\vartheta_{10}}
\]
and define a homomorphism
$$
\nu: \mathcal{C}^{\prime}(V_2^{\oplus 3}) \longrightarrow \mathcal{R}(\Gamma_0[2])[1/\chi_5]
$$
by making the substitution $a_i \leftrightarrow i$th component of $\chi_{2,-2}$, $b_i \leftrightarrow i$th component of $\chi^{(1)}_{2,0}$ and
$c_i \leftrightarrow i$th component of $\chi^{(2)}_{2,0}$. The criterion for holomorphicity allows a variant
for this case. Thus this easy example illustrates this map $\nu$.

\bigskip
The moduli space defined by the subgroup $\mathfrak{S}_{5}\times \mathfrak{S}_1$ of the
Galois group $\mathfrak{S}_6$, say fixing the sixth Weierstrass point, 
contains as a Zariski open subset the moduli space $\mathcal{M}_2[w]$ of curves
of genus $2$ with a marked Weierstrass point. We therefore denote $\mathcal{A}_2[2]/\mathfrak{S}_5$
 by $\mathcal{A}_2[w]$.
Analogously to  the description of $\mathcal{M}_2$ as a quotient stack we can describe $\mathcal{M}_2[w]$
as the $\mathfrak{S}_5$ quotient stack of $[\mathcal{P}\times V_{6,-2}^0/{\rm GL}_2]$ or directly as
$$
\mathcal{M}_2[w] \cong [(V_{5,-1} \times V_{1,-1})^0/{\GG}_m \times{\rm GL}_2]\, ,
$$
where $V_{a,b}={\rm Sym}^a(V)\otimes \det(V)^b$ and the action of ${\rm GL}_2$ is by
$$
(f_5,f_1) \mapsto (ad-bc)^{-1}\left( (cx+d)^5 f_5(\frac{ax+b}{cx+d}),(cx+d)f_1(\frac{ax+b}{cx+d})\right)
$$
and that of ${\GG}_m$  by
$$
(f_5,f_1)\mapsto (tf_5,t^{-1}f_1) \, .
$$
In $(V_{5,-1} \times V_{1,-1})^0$ the upper index $0$ indicates that we
consider pairs with non-vanishing discriminant of $f_5f_1$.
The stabilizer of a general point $(f_5,f_1)$ is generated by the
involution $(-1,-{\rm id})$.
We have a commutative diagram

\begin{displaymath}
\begin{xy}
\xymatrix{
\mathcal{M}_2[w] \ar[r] \ar[d] & [(V_{5,-1} \times V_{1,-1})^0/{\GG}_m \times{\rm GL}_2]\ar[d]^{\pi}\\
\mathcal{M}_2 \ar[r]^{\sim}  &  [V_{6,-2}^0/{\rm GL}_2]\\
}
\end{xy}
\end{displaymath}
where the arrow on the right sends $(f_5,f_1)$ to the product $f_5f_1$.
Just as before we get homomorphisms
$$
\mathcal{R}(\mathcal{A}_2[w]) \xrightarrow{\mu} \mathcal{C}^{\prime}(V^{\oplus 5} \oplus V)
\xrightarrow{\nu} \mathcal{R}(\mathcal{A}_2[w])[1/\chi_5]
$$
To describe $\nu$ explicitly we split the binary sextic as
$$
f_{6,-2}=\sum_{i=0}^6 a_i x_1^{6-i}x_2^i = (b_0 x_1^5+\cdots+ b_5x_2^5)(c_0x_1+c_1x_2)= f_{5,-1}f_{1,-1}\, .
$$
Covariants in $ \mathcal{C}^{\prime}(V^{\oplus 5} \oplus V)$ can be viewed as polynomials in $x_1,x_2$ with coefficients that are polynomials in $b_0,\ldots,b_5$ and $\gamma_0,\gamma_1$.
We define 
$$
\chi_{5,-5/3}= {\rm Sym}^5(G_1,\ldots,G_5)/\chi_5^{5/6}, \qquad \chi_{1,-1/3}=G_6/\chi_5^{1/6} \, .
$$
These are vector-valued meromorphic functions on the Siegel upper half space $\mathfrak{H}_2$ that we can write as
$$
\chi_{5,-5/3}= \sum_{i=0}^5 \beta_i\, x_1^{5-i}x_2^i, \qquad \chi_{1,-1/3}=\gamma_0 x_1 +\gamma_1x_2\, .
$$
\begin{proposition}
The homomorphism $\nu: \mathcal{C}^{\prime}(V^{\oplus 5} \oplus V)
\xrightarrow{\nu} \mathcal{R}(\mathcal{A}_2[w])[1/\chi_5]$
 is given by substituting $\beta_i$ for $b_i$ and $\gamma_i$ for $c_i$
in a covariant.
\end{proposition}

\begin{remark}
Alternatively, we could consider a splitting of 
$ \chi_{6,-2}= {\rm Sym}^6(G_1,\ldots,G_6)/\chi_5$
as $\chi_{6,-2}=\chi_{5,-1}\chi_{1,-1}$ with
$$
\chi_{5,-1}= \frac{{\rm Sym}^5(G_1,G_2,G_3,G_4,G_5)}{\vartheta_1\vartheta_2\vartheta_5\vartheta_6\vartheta_7\vartheta_8\vartheta_9}
\quad \text{\rm and} \quad \chi_{1,-1}= \frac{G_6}{\vartheta_3\vartheta_4\vartheta_{10}}
$$
and use these for the substitution.
\end{remark}
\bigskip
We illustrate this by finding modular forms in the $s[5,1]$ isotypic component of $S_{2,11}(\Gamma_2[2])$.
The generating series of the multiplicities $m(k)$ of the isotypic component $s[5,1]$ in $S_{2,2k+1}(\Gamma_2[2])$ is given by
$$
\sum_{k\geqslant 0} m(k) t^{2k+1} =\frac{t^{9}+
2 t^{11}+t^{13}+3 t^{15}+3 t^{17}+2 t^{19}+t^{21}+2 t^{23}}{\left(1-t^{4}\right) \left(1-t^{6}\right) \left(1-t^{10}\right) \left(1-t^{12}\right)}=t^{9}+2 t^{11}+2 t^{13}+\ldots
$$
and since there is no cusp form of weight $(2,11)$ in level $1$, we have
$\dim S_{2,11}(\Gamma_2[w])=2$.
We consider the following covariant
\[
\begin{aligned}
C_{2,2}=& 50\,((f_5,f_5)_4,l^2)_1\\
 = & (-50 a_{0}^{2} e_{0} e_{5}+6 a_{0}^{2} e_{1} e_{4}-a_{0}^{2} e_{2} e_{3}+20 a_{0} a_{1} e_{0} e_{4}-8 a_{0} a_{1} e_{1} e_{3}+3 a_{0} a_{1} e_{2}^{2})x_1^2+\ldots
\end{aligned}
\]
where we put
\[
f_5=e_0x_1^5+e_1x_1^4x_2+\ldots+e_5x_2^5 \quad \text{and} \quad  l=a_0x_1+a_1x_2. 
\]
The covariant $C_{2,2}=\sum_{j=0}^2 C_{2,2}^{(j)}x_1^{2-j}x_2^j$ is 
of degree $2$ in the coefficients of $f_5$, $l$ and in $x_1, x_2$ 
so the substitution $\nu$ provides a meromorphic modular form of weight 
$(2,1)$ in $s[5,1]\subset\Gamma_2[2]$ which is $\mathfrak{S}_5\times \mathfrak{S}_1$-invariant; 
that is,  a meromorphic modular form of weight $(2,1)$ on $\Gamma_2[w]$.
We then write $f_5=l_1l_2l_3l_4l_5$ and $l=l_6$ and check that for any partition $\pi$ we have
$$
w_{\pi}(C_{2,2})= 0, \quad v_{\pi}(C_{2,2})=-2\, .
$$
Thus the covariant $\nu(I^2_5C_{2,2})$ defines a modular form of weight $(2,11)$ on $\Gamma_2[w]$.
 Note that the orbit of $C_{2,2}$ under $\mathfrak{S}_6$ generates one of the $2$ copies of the irreducible 
representation $s[5,1]\subset S_{2,11}(\Gamma_2[2])$.

\end{section}


\begin{thebibliography}{99}

\bibitem{BC} J.\ Bergstr\"om, F.\ Cl\'ery:
{\sl Dimension formulas for spaces of vector-valued Siegel modular forms of degree $2$
and level $2$.} Publ.\ Mat., \textbf{69} (2), 367–-388, 2025.

\bibitem{BGHZ} J.\ Bruinier, G.\ van der Geer, G.\ Harder, D.\ Zagier:
{\sl The 1-2-3 of modular forms.} Universitext. Springer Verlag 2007.

\bibitem{CFG1} F.\ Cl\'ery, C.\ Faber, and G.\ van der Geer:
{\sl Covariants of binary sextics and vector-valued Siegel modular
forms of genus $2$.} Math.\ Annalen \textbf{369}(3–-4), 1649–-1669
(2017).

\bibitem{CFG2} F.\ Cl\'ery, C.\ Faber, and G.\ van der Geer:
{\sl Covariants of binary sextics and modular forms of degree $2$ with a character.}
Mathematics of Computation  \textbf{88}, No.\ 319, 2423--2441, 2019.

\bibitem{CFG3} F.\ Cl\'ery, C. Faber, and G. van der Geer:
{\sl Concomitants of ternary quartics and vector-valued Siegel modular
and Teichmueller modular forms of genus three.}
Selecta Mathematica (2020) 26:55.

\bibitem{CvdG-PAMQ} F. Cl\'ery and G. van der Geer:
{\sl Generating Picard modular forms by means of invariant theory.}
Pure and Applied Mathematics Quarterly
\textbf{19}, (2023), 95–-147, 2023

\bibitem{CvdG2023} F. Cl\'ery and G. van der Geer:
{\sl Modular forms via invariant theory.}
Research in Number Theory \textbf{9} (2023), No.\ 35, 10.

\bibitem{CvdGG} F.\ Cl\'ery, G.\ van der Geer,
S.\ Grushevsky: 
{\sl Siegel modular forms of genus $2$ and level $2$.}
International Journal of Mathematics \textbf{26}, No. 5 (2015) 1550034. 

\bibitem{D-O} I.\ Dolgachev, D.\ Ortland:
{\sl Point sets in projective spaces and theta functions.}
Ast\'erisque \textbf{165} (1988). Soci\'et\'e Math\'ematique de France.

\bibitem{F-SM2015}
E.\ Freitag, R.\ Salvati Manni:
{\sl Basic vector valued {S}iegel modular forms of genus two.}
Osaka Journal of Mathematics \textbf{52} (2015), 879--894.

\bibitem{vdG} G.\ van der Geer:
{\sl On the geometry of a Siegel modular threefold.}
Mathematische Annalen \textbf{260}, 317--350 (1982).

\bibitem{vdG-K} G.\ van der Geer, A.\ Kouvidakis:
{\sl Effective divisors on projectivized Hodge bundles and modular forms.}
Math.\ Nachr.\  \textbf{297} (2024), 1142--1170.


\bibitem{Geyer} 
W.\ Geyer: {\sl Invarianten bin\"arer Formen}.  In: Classification of Algebraic Varieties and Compact
Complex Manifolds. Lecture Notes in Math, vol.\ 412 (Springer, Berlin, 1974), pp. 36–-69.

\bibitem{Gordan} 
P.\ Gordan: {\sl Die simultanen Systeme bin\"arer Formen.} Math.\ Ann., 2(2):227--280, 1870.


\bibitem{HMSV} B.\ Howard, J.\ Millson, A.\ Snowden, R.\ Vakil:
{\sl A description of the outer automorphism of $S_6$ and the invariants of six points in projective space.}
Journal of Combinatorial Theory, Series A 115 (2008) 1296–-1303.


\bibitem{IgusaBook} J.-I. Igusa: {\sl Theta Functions.} 
Die Grundlehren der mathematischen Wissenschaften,
Vol.\ 194 (Springer-Verlag, 1972).

\bibitem{Igusa1964}
J.-I. Igusa:
{\sl On Siegel modular forms of genus two (II)},
Amer.\ J.\ Math.\ \textbf{86} (1964), 392--412.

\bibitem{Popoviciu} M.I.\ Popoviciu-Draisma:
{\sl Invariants of binary forms}. Inauguraldissertation Universit\"at Basel (2014).
{\url{https://edoc.unibas.ch/33424/1/thesis_popoviciudraisma.pdf}},

\bibitem{Springer} T.\ A.\ Springer: {\sl Invariant Theory}. 
Lecture Notes in Mathematics \textbf{585}. Springer-Verlag, Berlin-New York, 1977.

\bibitem{SF} J.\ Sylvester, F.\ Franklin:
{\sl Tables of the {G}enerating {F}unctions and {G}roundforms for
              the {B}inary {Q}uantics of the {F}irst {T}en {O}rders.}
American Journal of Mathematics \textbf{2} (1879), 223--251.

\bibitem{website} 
Tables available at {\url{https://smf.compositio.nl/Level2}}.
Tables for Siegel modular forms of degree $2$ and $3$,
initiated by Jonas Bergstr\"om, Carel Faber, and Gerard van der Geer, with
cooperation by Fabien Cl\'ery.




\end{thebibliography}
\end{document}